\documentclass[11pt]{smfart}
\usepackage{color}
\usepackage{amssymb,verbatim}
\usepackage{hyperref}
\usepackage{amsthm,array,amssymb,amscd,amsfonts,amssymb,latexsym, url}
\usepackage{amsmath}
\usepackage[all]{xy}
\usepackage[french]{babel}
\usepackage{amscd}

\newcounter{spec}
{\end{list}}

\renewcommand{\P}{{\mathbf P}}

\newcommand{\sX}{{\mathcal X}}

\newcommand{\N}{{\mathbb N}}
\newcommand{\Z}{{\mathbb Z}}
\newcommand{\Q}{{\mathbb Q}}
\newcommand{\C}{{\mathbb C}}
\newcommand{\R}{{\mathbb R}}

\renewcommand{\L}{{\mathbb L}}
\newcommand{\oi}{\hskip1mm {\buildrel \simeq \over \rightarrow} \hskip1mm}

\newcommand{\Br}{{\operatorname{Br   }}}
\newcommand{\Ker}{{\operatorname{Ker}}}

\newcommand{\Spec}{{\operatorname{Spec \ }}}

\renewcommand{\lim}{\varprojlim}

\def\oi{\hskip1mm {\buildrel \simeq \over \rightarrow} \hskip1mm}

\numberwithin{equation}{section}

\newfont{\gothic}{eufb10}

\renewcommand{\qed}{{\hfill$\square$}}

\newtheorem{theo}{Th\'{e}or\`{e}me}[section]
\newtheorem{prop}[theo]{Proposition}

\newtheorem{lem}[theo]{Lemme}
\newtheorem{cor}[theo]{Corollaire}
\theoremstyle{definition}
\newtheorem{defi}[theo]{D\'efinition}
\theoremstyle{remark}
\newtheorem{rema}[theo]{Remarque}

\newtheorem{ex}[theo]{Exemples}

\newcommand{\bthe}{\begin{theo}}
\newcommand{\ble}{\begin{lem}}
\newcommand{\bpr}{\begin{prop}}
\newcommand{\bco}{\begin{cor}}
\newcommand{\bde}{\begin{defi}}
\newcommand{\ethe}{\end{theo}}
\newcommand{\ele}{\end{lem}}
\newcommand{\epr}{\end{prop}}
\newcommand{\eco}{\end{cor}}
\newcommand{\ede}{\end{defi}}

\newcommand{\et}{{\operatorname{\acute{e}t}}}

\newcommand{\Pic}{\operatorname{Pic}}

\newcommand{\F}{{\mathbb F}}
\newcommand{\Gal}{{\rm Gal}}

\newcommand{\G}{{\mathbb G}}

\newcommand{\X}{{\mathcal X}}

\def\k{{\overline k}}
\DeclareFontFamily{U}{wncy}{}
\DeclareFontShape{U}{wncy}{m}{n}{%
<5>wncyr5%
<6>wncyr6%
<7>wncyr7%
<8>wncyr8%
<9>wncyr9%
<10>wncyr10%
<11>wncyr10%
<12>wncyr6%
<14>wncyr7%
<17>wncyr8%
<20>wncyr10%
<25>wncyr10}{}
\DeclareMathAlphabet{\cyr}{U}{wncy}{m}{n}

\begin{document}

  \title[Non rationalit\'e stable]{Non rationalit\'e stable sur les corps quelconques \\ Notes pour l'\'Ecole ``Birational geometry of hypersurfaces'' \\   Palazzo Feltrinelli, Gargagno del Garda \\  19--23 mars 2018 \\
   Version r\'evis\'ee,  24 avril 2018}

\author{J.-L. Colliot-Th\'el\`ene}
\address{Universit\'e Paris Sud\\Math\'ematiques, B\^atiment 425\\91405 Orsay Cedex\\France}
\email{jlct@math.u-psud.fr}

\date{24 avril 2018}
\maketitle

\section{Introduction}

Le texte ci-dessous a \'et\'e \'ecrit \`a l'occasion de l'\'Ecole ``Birational geometry of hypersurfaces''.
Il s'agit d'un rapport de synth\`ese.
Voici quelques points nouveaux.

Apr\`es  les articles initiaux de C.~Voisin \cite{voisinInv} et  de Colliot-Th\'el\`ene et  Pirutka \cite{CTP16},
les divers articles qui ont \'etabli la non rationalit\'e stable de divers types de vari\'et\'es classiques
ont utilis\'e la sp\'ecialisation de Fulton du groupe de Chow des z\'ero-cycles. 
Je d\'eveloppe dans ce texte une remarque de \cite{CTP16} : 
on peut remplacer la sp\'ecialisation du groupe de Chow des z\'ero-cycles
 par la sp\'ecialisation de la ${\rm R}$-\'equivalence.
On comparera ainsi la proposition \ref{tot} (due \`a Totaro \cite{T}) avec  les propositions \ref{remgen} et  \ref{Rtot}
et le th\'eor\`eme \ref{specRetdiff}.

La proposition \ref{H3Bettitors} (c) est nouvelle.
La proposition \ref{H1nul} am\'eliore un \'enonc\'e publi\'e dans \cite{ctJAG}.
La proposition \ref{fibrereductible} est nouvelle.

  \section{Entre rationalit\'e et unirationalit\'e}

\begin{lem}  Soit $k$ un corps.
Soit $X$ une $k$-vari\'et\'e g\'eom\'etriquement int\`egre
de dimension $d$. Consid\'erons les propri\'et\'es suivantes.

(i) La $k$-vari\'et\'e  $X$ est $k$-rationnelle, i.e. $k$-birationnelle  \`a $\P^d_{k}$.

(ii) La $k$-vari\'et\'e $X$ est stablement $k$-rationnelle, i.e. il existe un entier $n\geq 0$
tel que $X \times_{k} \P^n_{k}$ est $k$-birationnelle  \`a $\P^{n+d}_{k}$.

(iii) La $k$-vari\'et\'e  $X$ est facteur direct d'une vari\'et\'e $k$-rationnelle,
  c'est-\`a-dire qu'il existe   une $k$-vari\'et\'e $Y$ g\'eom\'etriquement  int\`egre  telle
que $X \times_{k} Y$ est $k$-birationnelle \`a un espace projectif.

 (iv)  La $k$-vari\'et\'e $X$ est r\'etractilement $k$-rationnelle, 
 c'est-\`a-dire qu'il existe un ouvert de Zariski non vide $U \subset X$, un ouvert de Zariski $V \subset \P^n_{k}$, et des $k$-morphismes
 $f : U \to V$ et $g: V \to U$ dont le compos\'e $g \circ f$ est l'identit\'e de $U$.

 (v) La $k$-vari\'et\'e $X$ est $k$-unirationnelle,  c'est-\`a-dire qu'il existe $ m\geq n$ et  une 
 $k$-application rationnelle dominante  $\P^m_{k}$ vers $X$.
 
On a :  (i) implique (ii) qui implique (iii), et (iv) implique (v). Si $k$ est infini, $(iii)$ implique (iv).
\end{lem}

\medskip

On sait que (ii) n'implique pas (i), m\^eme sur $k=\C$ (Beauville, CT, Sansuc, Swinnerton-Dyer \cite{BCTSSD}).
Sur un corps $k$ non alg\'ebriquement clos convenable, on sait montrer que (iii)
n'implique pas (ii). On ne sait pas ce qu'il en est sur $k=\C$.
On ne sait pas si une vari\'et\'e r\'etractilement rationnelle est facteur direct d'une vari\'et\'e rationnelle, m\^eme sur le corps des complexes.
Pour $p$ un nombre premier, et $PGL_{p} \subset GL_{N}$ un plongement de groupes,
Saltman a montr\'e que le quotient $GL_{N}/PGL_{p} $ est r\'etractilement rationnel.
On ne sait pas si cette vari\'et\'e est facteur direct d'une vari\'et\'e rationnelle.
Pour $H \subset G$ des groupes r\'eductifs connexes sur $\C$, on ne sait pas si $G/H$
est r\'etractilement rationnel.
 
 \medskip
 
 Exercice : Si $k$ est infini, l'hypoth\`ese de (v) implique la m\^eme hypoth\`ese
 avec $m=n$.

\begin{rema}
 Sur un corps $k$ alg\'ebriquement clos de caract\'eristique quelconque,
une $k$-vari\'et\'e projective et lisse $X$ r\'etractilement rationnelle est clairement rationnellement connexe
par cha\^{\i}nes. En caract\'eristique z\'ero, elle est donc s\'eparablement rationnellement connexe,
i.e. il existe un $k$-morphisme $f : \P^1 \to X$ tel que $f^*T_{X}$ soit un fibr\'e vectoriel ample.
\end{rema}

\section{Invariants birationnels stables}

\subsection{R-\'equivalence}

\begin{defi}  (Manin \cite{manin})
Soient $k$ un corps et $X$ une $k$-vari\'et\'e.
On dit que deux $k$-points $A,B \in  X(k)$ sont  \'el\'ementairement ${\rm R}$-li\'es s'il existe
un ouvert $U \subset \P^1_{k}$ et un $k$-morphisme $h: U \to X$
tel que $A, B$ soient dans $h(U(k))$. On dit que deux points 
 $A,B \in  X(k)$ sont ${\rm R}$-\'equivalents s'il existe une cha\^{i}ne 
 $A=A_{1}, A_{2}, \dots, A_{n}=B$ de $k$-points avec 
 $A_{i}$ et $A_{i+1}$ \'el\'ementairement ${\rm R}$-li\'es.
 On note $X(k)/{\rm R}$ le quotient de $X(k)$ par cette relation d'\'equivalence.
 \end{defi}
 
  Si $X$ est propre sur $k$, dans la d\'efinition ci-dessus, on peut prendre simplement $U=\P^1_{k}$.

Si $f : X \to Y$ est un $k$-morphisme, on a une application induite
$X(k)/{\rm R} \to Y(k)/{\rm R}$.  

Si $X$ est un ouvert d'un espace projectif $\P^n_{k}$, comme par deux $k$-points
il passe une droite $\P^1_{k}$, deux $k$-points quelconques de $X$
sont \'el\'ementairement R-li\'es, et $X(k)/{\rm R}$ a au plus un \'el\'ement.

\begin{defi}
Soient $k$ un corps et $X$ une $k$-vari\'et\'e int\`egre.

(i) On dit que $X$ est ${\rm R}$-triviale si, pour tout corps $F$ contenant $k$,
le quotient $X(F)/{\rm R}$ est d'ordre 1.

(ii)  On dit que $X$ est presque ${\rm R}$-triviale
s'il existe un ouvert de Zariski dense $U \subset X$ tel que,
pour tout corps $F$ contenant $k$, l'image de $U(F)$ dans $X(F)/{\rm R}$
est d'ordre 1.
\end{defi}

\begin{prop}
Soient $k$ un corps et $X$ une $k$-vari\'et\'e int\`egre, de corps des fonctions $F=k(X)$
et de point g\'en\'erique $\eta$.
Si $X$ est presque ${\rm R}$-triviale, alors il existe un $k$-point $m\in X(k)$
tel que sur $X_{F}$, le point g\'en\'erique $\eta \in X(F)$ et le point $m_{F} \in X(F)$
soient \'el\'ementairement li\'es. \qed
\end{prop}

\begin{rema}
Dans la suite de ce texte on sera int\'eress\'e \`a la notion de presque $R$-trivialit\'e dans une situation 
o\`u $U$ est lisse connexe mais o\`u $X$ n'est pas n\'ecessairement lisse.  On prendra garde qu'en l'absence
de lissit\'e de $X$
la condition de presque ${\rm R}$-trivialit\'e n'est  a priori pas tr\`es forte. 
Soit $Y \subset \P^n_{k}$ une $k$-vari\'et\'e quelconque et $X \subset \P^{n+1}_{k}$
le c\^{o}ne sur $Y$. Alors $X(k)/R=1$ car tout $k$-point de $X$ est \'el\'ementairement
R-li\'e au sommet $O \in X(k)$ du c\^{o}ne. 
Soit $U \subset X$ le compl\'ementaire du sommet du c\^{o}ne. Si par exemple $k=\C$, 
$Y \subset \P^2_{\C}$ est une courbe elliptique $E$,  alors
$U(\C)/R $ est en bijection avec $E(\C)/R=E(\C)$,
mais  l'application $U(\C) \to X(\C)/R$ a pour image un point.
 Les $\C$-vari\'et\'es $U$ et $X$ ne sont pas r\'etractilement rationnelles.
\end{rema}
\begin{defi}
Soient $k$ un corps et $f : X \to Y$ un $k$-morphisme.
On dit que $f$ est ${\rm R}$-trivial si pour tout corps $F$ contenant $k$,
l'application induite $X_{F}(F)/{\rm R} \to Y_{F}(F)/{\rm R}$ est une bijection.
\end{defi}
Un exemple est fourni par l'\'eclatement $X \to Y$ d'une sous-$k$-vari\'et\'e ferm\'ee lisse
dans une $k$-vari\'et\'e lisse $Y$.

\medskip

On a l'\'enonc\'e simple mais efficace suivant.
\begin{prop}\label{Rsimple}
Soit  $k$ un corps infini. 
Soit $X$ une $k$-vari\'et\'e int\`egre.
Si $X$ est r\'etractilement rationnelle, alors il existe un
 ouvert non vide $U \subset X$ tel que,
  pour tout corps $F$ contenant $k$,
 tout couple de points 
  $A, B \in U(F)$ est \'el\'ementairement ${\rm R}$-li\'e dans $U(F)$,
  et a fortiori dans $X(F)$.
  En particulier $X$ est presque ${\rm R}$-triviale.$\Box$
\end{prop}

\begin{theo}  Soit $k$ un corps de caract\'eristique z\'ero. 
Soient $Y$ et $X$ deux $k$-vari\'et\'es projectives et lisses
g\'eom\'etriquement int\`egres.
S'il existe un ouvert $Y' \subset Y$, un $k$-morphisme dominant  
 $Y' \to X$ et une $k$-section rationnelle de $Y' \to X$,
 alors il existe une application surjective $Y(k)/{\rm R} \to X(k)/{\rm R}$.
En particulier, si  $X$ est  r\'etractilement rationnelle, 
par exemple si $X$ est stablement $k$-rationnelle,
alors  $X$ est ${\rm R}$-triviale.
\end{theo}
\begin{proof}
On commence par \'etablir que si $Y \to X $ est l'\'eclat\'e d'une
 sous-$k$-vari\'et\'e lisse $Z$ dans une 
une $k$-vari\'et\'e projective et lisse $X$, l'application induite
$Y(k)/{\rm R}  \to X(k)/{\rm R}$ est une bijection.

Soit $Y ... \to X$ une $k$-application rationnelle dominante poss\'edant une section rationnelle.
D'apr\`es Hironaka, par \'eclatements successifs au-dessus de $Y$  le long de sous-$k$-vari\'et\'es ferm\'ees lisses, 
 on peut obtenir un $k$-morphisme $W \to X$ qui couvre l'application rationnelle
$ Y ... \to X$, et tel que l'application induite $W(k)/{\rm R} \to Y(k)/{\rm R}$ soit une bijection.
Le $k$-morphisme $W \to X$ admet une section $k$-rationnelle. Appliquant
le th\'eor\`eme de Hironaka \`a cette section, par \'eclatements successifs au-dessus de $X$
le long de sous-$k$-vari\'et\'es ferm\'ees lisses, on obtient une $k$-vari\'et\'e $Z$
muni d'une application $k$-birationnelle $f: Z \to X$ et d'un $k$-morphisme $Z \to W$
tel que le compos\'e $Z \to W \to X$ soit $f$.
L'application compos\'ee induite $Z(k)/{\rm R} \to W(k)/{\rm R} \to X(k)/{\rm R}$ est surjective,
donc aussi $W(k)/{\rm R} \to X(k)/{\rm R}$,
et l'application $W(k)/{\rm R} \to Y(k)/{\rm R}$ est une bijection.
\end{proof}

\medskip

Kahn et Sujatha \cite{KS} ont \'etabli  des extensions du th\'eor\`eme ci-dessus 
au-dessus d'un corps de caract\'eristique quelconque.

\begin{rema}
C'est une question ouverte si une $k$-vari\'et\'e  projective, lisse, connexe,  ${\rm R}$-triviale
est r\'etractilement rationnelle, et m\^eme si elle est facteur direct d'une 
$k$-vari\'et\'e $k$-rationnelle.

Le cas particulier suivant est d\'ej\`a tr\`es int\'eressant.
Soit $G$ un $k$-groupe alg\'ebrique (lin\'eaire) r\'eductif connexe.
 L'ensemble $G(k)/{\rm R}$
est alors naturellement muni d'une structure de groupe. 
Si $k$ est alg\'ebriquement clos, $G$ est une vari\'et\'e rationnelle. Sur un corps
$k$ quelconque, un $k$-groupe alg\'ebrique r\'eductif connexe $G$ est $k$-unirationnel.

Pour un tel $k$-groupe $G$, les questions suivantes
sont ouvertes. 
Sous des hypoth\`eses
particuli\`eres sur $k$ ou sur $G$, elles ont fait l'objet de nombreux
travaux \cite{Gillebki}.

(a) Le groupe $G(k)/{\rm R}$ est-il commutatif ?

(b) Si $G$ est ${\rm R}$-trivial, $G$ est-il r\'etractilement rationnel ?

(c) Si $k$ est un corps de type fini sur le corps premier, le groupe $G(k)/{\rm R}$
est-il fini ?
\end{rema}

\begin{rema}
Une  $k$-vari\'et\'e $X$ propre, lisse, connexe,  presque ${\rm R}$-triviale est  g\'eom\'e\-tri\-quement
rationnellement connexe par arcs (au sens de  Koll\'ar, Miyaoka, Mori \cite{kollar}).
Si $k$ est de caract\'eristique z\'ero,
elle est donc g\'eom\'etriquement s\'eparablement rationnellement connexe :
apr\`es extension du corps de base, 
il existe un   morphisme $f: \P^1  \to X$ tel que
$f^* T_{X} $ soit un fibr\'e vectoriel ample.
D\'eterminer ce qu'il en est en caract\'eristique positive.
\end{rema}

\subsection{Groupe de Chow des z\'ero-cycles}

Soit $X$ une $k$-vari\'et\'e alg\'ebrique. 
On note $Z_{0}(X)$ le groupe ab\'elien libre sur les points ferm\'es de $X$. 
On a l'application degr\'e $deg_{k} : Z_{0}(X) \to \Z$ envoyant 
$\sum_{i}n_{i}P_{i}$ sur $\sum_{i} n_{i} [k(P_{i}):k]$.
Pour tout $k$-morphisme $f: Y \to X$ de $k$-vari\'et\'es, on dispose d'une application
induite $f_{*} : Z_{0}(X) \to Z_{0}(Y)$ qui est additive et envoie le point ferm\'e $P \in X$
d'image le point ferm\'e $Q$ de $Y$ sur
 $[k(P):k(Q)]Q$.   Cette application pr\'eserve le degr\'e.

Si $C \to X$ est un $k$-morphisme propre d'une $k$-courbe normale int\`egre $C$
et $g\in k(C)^*$ une fonction rationnelle sur $C$, on leur associe le z\'ero-cycle
$f_{*}(div_{C}(g))$. On d\'efinit $CH_{0}(X)$ comme le quotient de $Z_{0}(X)$ par
le sous-groupe engendr\'e par tous les $f_{*}(div_{C}(g))$ pour $C, g, f$ comme ci-dessus.

Si  la $k$-vari\'et\'e $X$ est propre, le  degr\'e   $deg_{k} : Z_{0}(X) \to \Z$,
induit un homomorphisme $deg_{k} :  CH_{0}(X) \to \Z$,
car le degr\'e du diviseur  des z\'eros d'une fonction rationnelle sur une courbe propre
est z\'ero.

Plus g\'en\'eralement, pour $ h : Y \to X$ un $k$-morphisme propre, l'application 
$f_{*} : Z_{0}(X) \to Z_{0}(Y)$ induit une application $f_{*}: CH_{0}(Y) \to CH_{0}(X)$.

\begin{defi}
Soit  $X$ une $k$-vari\'et\'e propre. On dit que $X$ est (universellement) $CH_{0}$-triviale si
pour tout corps $F$ contenant $k$, le degr\'e $$deg_{F} : CH_{0}(X_{F}) \to \Z$$
est un isomorphisme.
\end{defi}

\begin{prop}(Merkurjev) \cite[Thm. 2.11]{merk} \label{merku}
Soit  $X$ une $k$-vari\'et\'e propre, lisse,  g\'eom\'etriquement int\`egre. 
Les propri\'et\'es suivantes sont \'equivalentes :

(i)  La $k$-vari\'et\'e $X$ est   $CH_{0}$-triviale.

(ii) $X$ poss\`ede  un z\'ero-cycle de degr\'e 1 et,
pour $F=k(X)$ le corps des fonctions de $X$, l'application $deg_{F} : CH_{0}(X_{F}) \to \Z$ est un isomorphisme.

(iii) La classe du point g\'en\'erique de $X$ dans $CH_{0}(X_{k(X)})$  est dans l'image de l'application
image r\'eciproque  $CH_{0}(X) \to CH_{0}(X_{k(X)})$.
\end{prop}

 \begin{defi}
Soit  $f : Y \to X$ un $k$-morphisme propre de $k$-vari\'et\'es. 
On dit que $f$ est un $CH_{0}$-isomorphisme (universel) si 
pour tout corps $F$ contenant~$k$, 
l'application induite  $$f_{F,*}: CH_{0}(Y_{F}) \to CH_{0}(X_{F})$$
est un isomorphisme.
\end{defi}

\begin{lem}\label{gll}  \cite[Cor. 6.7]{GLL}
Soit $X$ une vari\'et\'e quasi-projective  r\'eguli\`ere   connexe  sur un corps    $k$.
\'Etant donn\'e un z\'ero-cycle $z$ sur  $X$ et un ouvert de Zariski  non vide $U \subset X$,
il existe un z\'ero-cycle $z'$ sur $X$ dont le support est dans $U$
 et qui est rationnellement
\'equivalent \`a $z$ sur $X$.
\end{lem}

\begin{lem}\label{RtriviCH_{0}triv}
 Soient $k$ un corps et $X$ une $k$-vari\'et\'e projective, lisse, g\'eom\'etriquement int\`egre.
Si $X$ est  presque ${\rm R}$-triviale, alors $X$ est $CH_{0}$-triviale. $\Box$
\end{lem}
\begin{proof}
Ceci r\'esulte du lemme \ref{gll} (voir la d\'emonstration de \cite[Lemme 1.5]{CTP16}).
\end{proof}

\begin{prop}  \cite[Lemme 1.5]{CTP16}
Soit $k$ un corps.
Soit $X$ une $k$-vari\'et\'e projective et lisse g\'eom\'etriquement int\`egre.
Si $X$ est r\'etractilement rationnelle, alors  $X$ est  $CH_{0}$-triviale.
\end{prop}
\begin{proof}   Pour $k$ un corps infini,
ceci r\'esulte imm\'ediatement de la  proposition \ref{Rsimple} et du lemme \ref{RtriviCH_{0}triv}.
Le cas d'un corps fini s'\'etablit par un argument de normes.
\end{proof}

\begin{rema}
On prendra garde qu'il existe des surfaces connexes projectives et lisses  sur $\C$
qui sont $CH_{0}$-triviales mais sont de type g\'en\'eral, et donc ne sont pas
rationnellement connexes, et donc pas ${\rm R}$-triviales
 (Voisin ; \cite[Prop. 1.9]{ACTP}). De telles surfaces satisfont $H^0(X, \Omega^{i})=0$
 pour $i=1,2$
 (Prop. \ref{tot} ci-dessous) mais ne satisfont pas $H^0(X, (\Omega^{2})^{\otimes 2})=0$.
 \end{rema}
 
\subsection{Action du groupe de Galois sur le groupe de Picard}\label{galpic}

Soit $X$ une vari\'et\'e projective lisse, connexe, g\'eom\'etriquement rationnellement connexe
sur un corps $k$. Si $k$ est alg\'ebriquement clos de car. z\'ero, tout rev\^etement fini galoisien
\'etale connexe est trivial (Koll\'ar, Miyaoka, Mori).
L groupe $Pic(X)=H^1_{Zar}(X,\G_{m})=H^1_{\et}(X,\G_{m})$ est un groupe ab\'elien
libre de type fini.  Il n'y a pas l\`a d'invariant qui d\'etecterait la non rationalit\'e.
Si $k$ n'est pas alg\'ebriquement  clos, la situation change.

Les invariants suivants ont tout d'abord  \'et\'e \'etudi\'es  par Shafarevich, Manin \cite{manin}, Iskovskikh,  Voskresenski\u{\i}.

\begin{theo}\label{picpermut}
Soient $k$ un corps, $k^s$ une cl\^oture s\'eparable, et $g=\Gal(k^s/k)$. On note $X^s=X\times_{k}k^s$.
Soient $X$ et $Y$ deux $k$-vari\'et\'es propres, lisses, g\'eom\'etriquement int\`egres.

(a) Si $X$ est $k$-birationnelle \`a $Y$, alors il existe des $g$-modules de permutation de type
fini $P_{1}$ et $P_{2}$ et un isomorphisme de modules galoisiens
$$ \Pic(X^s) \oplus P_{1} \simeq \Pic(Y^s) \oplus P_{2},$$
et l'on a $H^1(k, \Pic(X^s)) \simeq H^1(k, \Pic(Y^s) )$.

(b)  Supposons ${\rm car}(k)=0$. Si $X$ est $CH_{0}$-triviale, 
alors
le module galoisien $ \Pic(X^s) $ est un facteur direct d'un 
$g$-module de permutation de type
fini,  et,  pour tout corps $F$ contenant $k$, on a $H^1(F, \Pic(X\times_{F}F^s))=0$.

(c) Supposons ${\rm car}(k)=0$. Si $X$ est r\'etractilement rationnelle, alors
le module galoisien $ \Pic(X^s) $ est un facteur direct d'un 
$g$-module de permutation de type
fini, et,  pour tout corps $F$ contenant $k$, on a $H^1(F, \Pic(X\times_{F}F^s))=0$.
\end{theo}
\begin{proof}
Pour l'\'el\'egante d\'emonstration de (a) due \`a L. Moret-Bailly, voir \cite[Prop. 2A1, p. 461]{CTSa87a}.
Pour  (b), voir CT, Appendice \`a un article de S. Gille (J. Algebra 440 (2015) 443--463). On en d\'eduit alors (c).
\end{proof}

\begin{rema}
Pour toute $k$-vari\'et\'e projective, lisse, g\'eom\'etriquement connexe $X$
avec un $k$-point, notant $X^s=X\times_{k}k^s$,
on a une suite exacte
$$ 0  \to \Br(k) \to \Ker[\Br(X) \to \Br(X^s)] \to H^1(k,\Pic(X^s)) \to 0.$$

On voit donc que l'invariant ``module galoisien $\Pic(X^s) $ \`a addition pr\`es de module de permutation''
raffine le sous-groupe ``alg\'ebrique'' $\Ker[\Br(X) \to \Br(X^s)] $ du groupe de Brauer de~$X$.
\end{rema}

Voici un exemple d'application.

\begin{prop}\label{diagcub}
Soit $k$ un corps de caract\'eristique diff\'erente de 3. Soient $a,b,c,d \in k^*$.
Si aucun des quotients $ab/cd$ n'est un cube dans $k^*$, alors
la $k$-surface cubique $X \subset \P^3_{k}$ d'\'equation
$$ax^3+by^3+cz^3+dt^3=0$$
n'est pas stablement $k$-rationnelle, et si de plus $k$ est de caract\'eristique \'z\'ero,
elle n'est pas r\'etractilement rationnelle.
\end{prop}

Le cas $a=b=c=1$ est trait\'e dans le livre de Manin \cite{manin} et dans \cite{CTSa87a}.
 Le cas g\'en\'eral
est fait dans un article de CT-Kanevsky-Sansuc.

On a une r\'eciproque : si $ab/cd$ est un cube, et $X$ poss\`ede un $k$-point,
alors $X$ est $k$-birationnelle \`a $\P^2_{k}$.

Comme on verra plus bas, l'invariant discut\'e ici est tr\`es int\'eressant pour certaines classes
de vari\'et\'es g\'eom\'etriquement rationnelles, par exemple les surfaces. Mais si $X$ est une hypersurface lisse de degr\'e $d \leq n$
dans $\P^n_{k}$ avec $n \geq 4$, alors $\Z= \Pic(\P^n_{k^s}) \oi \Pic(X^s)$ et ceci ne donne aucune information
sur l'\'eventuelle non $k$-rationalit\'e de $X$.

\subsection{Cohomologie non ramifi\'ee}

R\'ef\'erences : \cite{ctoj},  \cite{ctbarbara}, \cite{P}

Soit $A$ un anneau de valuation discr\`ete, $K$ son corps
des fractions, $\kappa_{A}$ son corps r\'esiduel. Soit $n>1$ un entier inversible dans $\kappa_{A}$.
Pour tous entiers $j \in \Z$ et $i\geq 1$, on d\'efinit d'une application r\'esidu
entre groupes de cohomologie galoisienne
$$ \partial_{A} : H^{i}(K,\mu_{n}^{\otimes j} ) 
\to H^{i-1}  (\kappa_{A},
\mu_{n}^{\otimes j-1} ).
$$
Pour $k$ un corps, $X$ une $k$-vari\'et\'e    lisse connexe de corps des fonctions $k(X)$ et $n>0$
entier premier \`a la caract\'eristique de $k$,
on d\'efinit
$$H^i_{nr}(X/k, \mu_{n}^{\otimes j}) = \cap_{x \in X^{(1)}}
[\Ker \partial_{x } : H^i(k(X), \mu_{n}^{\otimes j})  \to H^{i-1}(k(x),\mu_{n}^{\otimes j-1})].$$
Ici $x$ parcourt les points de codimension 1 de $X$ et $k(x)$ est le corps r\'esiduel en $x$.

Soit ${\mathcal H}^{i}_{X}(\mu_{n}^{\otimes j})$ le faisceau Zariski sur $X$ associ\'e au
pr\'efaisceau $U \mapsto   H^{i}(U, \mu_{n}^{\otimes j})$.

\bigskip

La conjecture de Gersten pour
la cohomologie \'etale  (th\'eor\`eme de Bloch-Ogus--Gabber,  voir \cite{CTKH})   implique :

$\bullet$ Le faisceau  Zariski ${\mathcal H}^{i}_{X}(\mu_{n}^{\otimes j})$ est un sous-faisceau du faisceau 
constant d\'efini par $H^{i}(k(X),\mu_{n}^{\otimes j})$. C'est   un  foncteur contravariant sur la cat\'egorie
des $k$-vari\'et\'es propres, lisses, connexes.

$\bullet$ On a $H^0(X,{\mathcal H}^{i}_{X}(\mu_{n}^{\otimes j})) = H^i_{nr}(X/k, \mu_{n}^{\otimes j}).$

$\bullet$ Si $X$ est propre, lisse, connexe, le groupe  $H^i_{nr}(X/k, \mu_{n}^{\otimes j})$ co\"{i}ncide avec
$$H^0(X,{\mathcal H}^{i}_{X}(\mu_{n}^{\otimes j})) \subset H^i_{nr}(X/k, \mu_{n}^{\otimes j}).$$

$\bullet$  Si $X$ est propre, lisse, connexe, ce groupe     co\"{i}ncide  avec
$$ \cap_{A}  [\Ker \partial_{A} : H^i(k(X), \mu_{n}^{\otimes j}) \to  H^{i-1}(\kappa_{A}),\mu_{n}^{\otimes j-1})],$$
o\`u $A$ parcourt tous les anneaux de valuation discr\`ete contenant $k$ et de corps des fractions
$k(X)$.

$\bullet$  Pour tout entier $m\geq 1$, on a
$$H^{i}(k,\mu_{n}^{\otimes j}) \oi H^{i}_{nr}(k(\P^m_{k})/k,\mu_{n}^{\otimes j}).$$

\bigskip

On a les propri\'et\'es suivantes :

$$H^1_{nr}(k(X)/k,\Z/n) = H^1_{\et}(X,\Z/n).$$

$$H^2_{nr}(k(X)/k,\mu_{n}) = \Br(X] [n],$$
o\`u $\Br(X)[n]$ est le sous-groupe de $n$-torsion du groupe de Brauer
$\Br(X)=H^2_{\et}(X,\G_{m})$ de la $k$-vari\'et\'e $X$.

 Le groupe $\Br(X)$ est un invariant $k$-birationnel stable des $k$-vari\'et\'es propres 
et lisses (Grothendieck, Hoobler, Gabber, \v{C}esnavi\v{c}ius).

Les groupes $H^i_{nr}(X/k, \mu_{n}^{\otimes j}) $ sont fonctoriels contravariants
pour les $k$-morphismes  quelconques de $k$-vari\'et\'es propres, lisses, connexes.
Ceci r\'esulte de la formule
$$H^i_{nr}(X/k, \mu_{n}^{\otimes j})= H^0(X,{\mathcal H}^{i}_{X}(\mu_{n}^{\otimes j})).$$

En particulier pour toute $k$-vari\'et\'e $X$ propre, lisse, g\'eom\'etriquement connexe,
pour tout corps $F$ contenant $k$
 on dispose d'accouplements
$$X(F) \times H^i_{nr}(X_{F}/F, \mu_{n}^{\otimes j})  \to H^{i}(F,\mu_{n}^{\otimes j}) $$
qui passent au quotient par la ${\rm R}$-\'equivalence :
$$X_{F}(F)/{\rm R} \times H^i_{nr}(X_{F}/F, \mu_{n}^{\otimes j})  \to H^{i}(F,\mu_{n}^{\otimes j}).$$
 
 De fa\c con plus d\'elicate, pour toute   $k$-vari\'et\'e projective, lisse, connexe $X$, on dispose d'accouplements
 $$CH_{0}(X_{F}) \times  H^i_{nr}(X_{F}/F, \mu_{n}^{\otimes j})  \to H^{i}(F,\mu_{n}^{\otimes j}).$$
Pour une d\'emonstration d\'etaill\'ee dans le
 cas du groupe de Brauer, voir \cite{ABBB}.

 La proposition \ref{Rsimple} donne alors :
 \begin{prop}
Soit  $X$ une  $k$-vari\'et\'e propre  lisse g\'eom\'etriquement connexe.
  Si $X$ est  presque R-triviale, 
  pour tous $i,j$, et tout corps $F$ contenant $k$, on a
$$H^{i}(F,\mu_{n}^{\otimes j})
 \oi H^i_{nr}(X_{F}/F, \mu_{n}^{\otimes j})$$
 et
 $\Br(F) = \Br(X_{F})$.
 \end{prop}
 \begin{proof}
Pour le voir, il suffit de monter sur le corps $K=k(X)$ est d'utiliser le fait que le point g\'en\'erique
est ${\rm R}$-\'equivalent \`a un point de $X(k)\subset X(k(X))$.
\end{proof}

\begin{cor}
Soit  $X$ une  $k$-vari\'et\'e projective lisse g\'eom\'etriquement connexe.
  Si $X$ est r\'etractilement rationnelle,  alors,  pour tous $i,j$, 
 et tout corps $F$ contenant $k$, on a
$$H^{i}(F,\mu_{n}^{\otimes j})
 \oi H^i_{nr}(X_{F}/F, \mu_{n}^{\otimes j})$$
 et
 $\Br(F) = \Br(X_{F})$.
\end{cor}
 
Ceci r\'esulte de l'\'enonc\'e pr\'ec\'edent, sauf dans le cas d'un corps $k$  fini.  
Dans ce cas on monte sur des extensions finies de $k$ suffisamment grosses
et on utilise un argument de norme.

 \begin{prop}
Si une $k$-vari\'et\'e projective et lisse g\'eom\'etriquement connexe $X$ est   $CH_{0}$-triviale,
alors, pour tous $i,j$, alors,  pour tous $i,j$, 
 et tout corps $F$ contenant $k$, on a
$$H^{i}(F,\mu_{n}^{\otimes j})
 \oi H^i_{nr}(X_{F}/F, \mu_{n}^{\otimes j})$$
 et
 $\Br(F) = \Br(X_{F})$.
 \end{prop}
 
 Cet \'enonc\'e implique le pr\'ec\'edent, mais sa d\'emonstration est un peu plus \'elabor\'ee,
 car elle passe par l'accouplement avec le groupe de Chow.

 \bigskip
 
 On a des \'enonc\'es analogues aux pr\'ec\'edents en rempla\c cant les $H^{\bullet}(\F, \mu_{n}^{\otimes \bullet})$
 des  corps $F$  par les modules de cycles de Rost, par exemples par les groupes $K^{M}_{i}(F)$ ($i\in \N$)
 de $K$-th\'eorie de Milnor des corps. Voir \`a ce sujet l'article de Merkurjev \cite{merk}, qui
 montre que  a trivialit\'e universelle de tous les invariants non ramifi\'es de tous les modules de cycles de Rost
 pour une $k$-vari\'et\'e projective lisse
 connexe donn\'ee $X$ est \'equivalente au fait que cette vari\'et\'e est $CH_{0}$-triviale \cite[Thm. 2.11]{merk}.
 
\bigskip

\begin{rema}
Soit $k$ un corps alg\'ebriquement clos et $\ell$ un premier diff\'erent de la caract\'eristique de $k$.
Soit $X$ une $k$-vari\'et\'e projective, lisse, connexe, de dimension $d$.
Sous l'une des hypoth\`eses ci-dessus, on conclut  $H^1_{\et}(X,\mu_{{\ell}^n})=0$ pour tout entier $n>0$. 
On a donc $H^1_{\et}(X,\Z_{\ell}(1))=0$.
La suite de Kummer en
cohomologie \'etale (o\`u $\G_{m}  \to \G_{m}$ est donn\'e par $x \mapsto x^n$) :
$$ 1 \to \mu_{{\ell}^n} \to \G_{m}  \to \G_{m}  \to 1$$
donne d'abord que $\Pic(X)$ est sans torsion, donc \'egal au groupe de N\'eron-Severi $NS(X)$,
qui est donc lui-m\^eme sans torsion.
Elle donne ensuite
 des suites exactes courtes compatibles (en $n$) :
$$ 0 \to \Pic(X)/\ell^n \to H^2_{\et}(X,\mu_{{\ell}^n}) \to \Br(X)[\ell^n] \to 0.$$
En passant \`a la limite on obtient une suite exacte
$$ 0 \to NS(X)\otimes \Z_{\ell} \to H^2_{\et}(X,\Z_{\ell}(1)) \to T_{\ell}(\Br(X)) \to 0.$$
Un module de Tate $T_{\ell}(A)= {\rm lim proj}_{n} A[\ell^n]$ est toujours sans torsion. Ainsi $H^2_{\et}(X,\Z_{\ell}(1))$ est sans torsion.
Si $k=\C$, les th\'eor\`emes de comparaison donnent alors $H^1_{Betti}(X,\Z)=0$ et
$H^2_{Betti}(X,\Z)_{tors}=0$. On va voir ci-dessous que l'on a aussi $H^3_{Betti}(X,\Z)_{tors}=0$. 
Par diverses dualit\'es de Poincar\'e, ces r\'esultats impliquent $H^{2d-1}_{Betti}(X,\Z)=0$.
En utilisant la d\'ecomposition de la diagonale, un argument de correspondance et
des d\'esingularisations, C.~Voisin \'etablit ces r\'esultats. Elle \'etablit aussi
$H^{2n-2}_{Betti}(X,\Z)_{tors}=0$.
\end{rema}

 \subsection{Calcul du groupe de Brauer}
 
 \begin{prop}\label{H3Bettitors}
 Soit $k$ un corps alg\'ebriquement clos de car. z\'ero, et soit $X$
 une $k$-vari\'et\'e projective et lisse rationnellement connexe.

 (a)  Si $k=\C$, alors $\Br(X) \oi H^3_{Betti}(X(\C),\Z)_{tors}$.
 
 (b) En g\'en\'eral, $\Br(X)$ est un groupe fini isomorphe \`a $\oplus_{\ell} H^3_{\et}(X, \Z_{\ell}(1))$.
 
(c) Si $F$ est un corps qui contient $k$, l'application naturelle $\Br(X) \to \Br(X_{F} )/\Br(F)$
est un isomorphisme.
\end{prop}
 \begin{proof}
 
 En utilisant la suite de Kummer
 $$ 1 \to \mu_{{\ell}^n} \to \G_{m}  \to \G_{m}  \to 1$$
 en cohomologie \'etale,
 Grothendieck \cite{GGB} a montr\'e que pour toute $k$-vari\'et\'e  projective lisse connexe,
le groupe de Brauer de $X$ est une extension du groupe fini 
 $\oplus_{\ell} H^3_{\et}(X, \Z_{\ell}(1))$ par un groupe divisible.
 Si $X$ est rationnellement connexe, on montre qu'il existe un entier $N>0$
qui annule $A_{0}(X_{F})$ pour tout corps $F$ contenant $k$. Ceci implique
que le groupe divisible est annul\'e par $N$, donc est nul.
Ceci \'etablit (a) et (b). Pour (c), consid\'erons les inclusions
$k \subset F \subset \overline{F}$. 
On a les applications naturelles
$$ \Br(X) \to \Br(X_{F}) \to \Br(X_{\overline{F}}).$$
Fixons un $k$-point  $P$ de $X(k)$, et consid\'erons les sous-groupes
de ces divers groupes form\'es des \'el\'ements nuls  en $P$. On a alors les applications
$$ \Br^P(X) \to \Br^P(X_{F}) \to \Br^P(X_{\overline{F}}).$$
D'apr\`es (b), la compos\'ee est un isomorphisme. 
 Comme on a 
 $\Pic(X) = \Pic(X_{F}) = \Pic(X_{\overline{F}})$,  la proposition \ref{picpermut}
 donne 
 $$\Br(F) = \Ker [\Br(X_{F}) \to \Br(X_{\overline{F}})]$$ et donc
$ \Br^P(X_{F}) \hookrightarrow \Br^P(X_{\overline{F}}).$ On a donc $ \Br^P(X) = \Br^P(X_{F})$.
 \end{proof}

  Soit maintenant $k$ de caract\'eristique z\'ero quelconque. 
  
  Pour $X$ une $k$-vari\'et\'e projective et lisse
  g\'eom\'etriquement connexe,  avec $X(k) \neq \emptyset$,
au paragraphe \ref{galpic} on a vu comment calculer $$\Ker[ \Br(X) \to \Br(X\times_{k}\k)^{{\rm Gal}(\k/k   )}].$$
La question du calcul du groupe des invariants $\Br(X\times_{k}\k)^g$ est d\'elicate, c'est un probl\`eme ``arithm\'etique'',
nous n'en  parlerons pas
ici.

Pour $X$ une $\C$-vari\'et\'e projective, lisse, rationnellement connexe,
la formule $\Br(X) \oi H^3_{Betti}(X(\C),\Z)_{tors}$ donn\'ee ci-dessus est th\'eoriquement satisfaisante.
Mais en pratique, quand on se donne une vari\'et\'e concr\`ete, elle a tendance \`a \^etre singuli\`ere.
ll faudrait la d\'esingulariser, ce qui en grande dimension est difficile, en outre il faut ensuite calculer
sur un mod\`ele projectif et lisse le groupe $ H^3_{Betti}(X(\C),\Z)_{tors}$. C'est ce qu'avaient fait Artin et
Mumford \cite{artinmumford} pour une vari\'et\'e de dimension 3 fibr\'ee en coniques sur le plan projectif
complexe.

Dans \cite{ctoj}, on a donn\'e une autre fa\c con d'\'etablir $\Br_{nr}(\C(X)/\C) \neq 0$ pour
des fibrations en coniques $X$ sur le plan complexe.

Soit $k$ un corps, $k^s$ une  cl\^oture s\'eparable  et $X$ une $k$-vari\'et\'e projective et lisse
  g\'eom\'etriquement connexe. Si l'on omet l'hypoth\`ese  $X(k) \neq \emptyset$, on a une suite
  exacte
  $$ 0 \to \Pic(X) \to \Pic(X^s)^g \to \Br(k) \to \Ker[\Br(X) \to \Br(X^s)] \to H^1(k, \Pic(X^s)).$$
  
Si $X$ est une conique lisse $C$  sans $k$-point,  de corps des fonctions $k(C)$, on trouve une suite exacte
$$0 \to \Z/2 \to \Br(k) \to \Br(C) \to 0.$$
Si ${\rm car}(k) \neq 2$ et $C$ est donn\'ee par l'\'equation homog\`ene $x^2-ay^2-bz^2=0$,
le noyau de $\Br(k) \to \Br(C)$ -- qui est aussi le noyau de $\Br(k) \to \Br(k(C))$ car $\Br(C)$ s'injecte dans $\Br(k(C))$
puisque $C$ est lisse --
est donn\'e par la classe de l'alg\`ebre de quaternions $(a,b)$. Ce r\'esultat remonte \`a Witt,
et fut \'etendu aux vari\'et\'es de Severi-Brauer par F. Ch\^{a}telet.

Le point de vue ``birationnel'' adopt\'e par Ojanguren et moi dans \cite{ctoj} est le suivant.
On a une vari\'et\'e projective et lisse  (non explicite) $X$ sur $\C$ munie d'une fibration 
$p : X \to S=\P^2_{\C}$ dont la fibre g\'en\'erique est une conique $C/\C(S)$ sans point rationnel
(i.e. la fibration n'a pas de section rationnelle). La fibration d\'eg\'en\`ere le long d'une union finie
de courbes int\`egres  $D_{i}\subset S$.
On dispose de la classe $\alpha \in \Br(\C(S))$ de la conique g\'en\'erique, d'ordre 2,
non nulle, et qui forme exactement le noyau de l'application
$$\Br(\C(S)) \to \Br(\C(X)).$$ 
Comme $S=\P^2_{\C}$, on a $\Br(S)=0$, et l'application r\'esidu en tous les points de codimension 1 de $S$
donne une injection
$$ \delta :  \Br(\C(S)) \hookrightarrow  \oplus_{x \in S^{(1)}  } H^1(\C(x),\Q/\Z).$$
La classe $\alpha$ a un nombre fini de r\'esidus non triviaux, correspondant aux
points o\`u la fibration d\'eg\'en\`ere.
Sous des hypoth\`eses   sur la d\'eg\'en\'erescence,
on exhibe une autre classe $\beta \in \Br(\C(S))$ dont  le r\'esidu total $\delta(\beta)$
est non nul et form\'e d'un sous-ensemble propre des $\delta_{x}(\alpha)$.
Si le diviseur de d\'eg\'en\'erescence est une union de courbes lisses dont 
la r\'eunion est un diviseur \`a croisements normaux,
par comparaison avec les r\'esidus aux points de codimension 1 de $X$,t
ceci assure que $\beta$ devient non ramif\'e dans $\Br(\C(X))$ (sans calcul
explicite d'un mod\`ele projectif et lisse de $X$), et par ailleurs,
que $\beta$ n'est pas dans le noyau $\Z/2 = \Ker [\Br(\C(S) \to \Br(\C(X))].$
Ainsi $\Br_{nr}(\C(X)) \neq 0$, et  la $\C$ vari\'et\'e $X$ n'est pas r\'etractilement rationnelle,
ni m\^{e}me $CH_{0}$-triviale.

\subsection{Calcul de la  cohomologie non ramifi\'ee de degr\'e sup\'erieur}

 Pour $X$ une vari\'et\'e projective lisse rationnellement connexe sur $\C$,    
  on ne dispose pas  pour les invariants cohomologiques sup\'erieurs $H^{i}_{nr}(\C(X)/\C, \Q/\Z)$, $i \geq 3$,
d'un analogue des diff\'erents \'enonc\'es de  la proposition  \ref{H3Bettitors}.
De fait il est peu probable que ces invariants soient constants dans une famille
projective et lisse de telles vari\'et\'es (voir \cite{CTVoisin} pour une discussion).

Pour $X$ comme ci-dessus, on a un  certain nombre de r\'esultats  int\'eressants
en degr\'e $i=3$, et quelque r\'esultats  en degr\'e $i>3$.
Je renvoie ici le lecteur aux travaux \cite{CTVoisin}, \cite{voisinInv} et \cite{CTpourmerk}.
Dans \cite{CTVoisin}, avec C.~Voisin,
on \'etablit un 
lien entre $H^{3}_{nr}(\C(X)/\C, \Q/\Z)$ et  la conjecture de Hodge enti\`ere pour les cycles
de codimension 2. 

\medskip

 Le cas des hypersurfaces cubiques dans $\P^n_{\C}$, $n \geq 4$,
a \'et\'e particuli\`erement \'etudi\'e, en particulier par C. Voisin \cite{voisinInv},
voir aussi \cite{CTpourmerk}.
Pour de telles hypersurfaces, on a $H^{3}_{nr}(\C(X)/\C, \Q/\Z)=0$. 
Pour $F$ un corps contenant $\C$,
on sait que l'application
$$ H^3(F,\Q/\Z) \to H^3_{nr}(F(X)/F,\Q/\Z)$$
est un isomorphisme pour $n\geq 5$. Pour $n=4$, la question est ouverte.

\bigskip

Le point de vue birationnel adopt\'e dans \cite{ctoj} pour revisiter l'exemple d'Artin et Mumford 
repose sur le fait que sur un corps $k$ de car. diff\'erente de 2, et pour une conique $C$
sur $k$ d'\'equation homog\`ene $x^2-ay^2-bz^2=0$, avec $a,b \in k^*$, le noyau
de l'application
$H^2(k,\Z/2) \to H^2(k(C), \Z/2)$ est d'ordre au plus 2, engendr\'e par la classe
de l'alg\`ebre  de quaternions $(a,b)$, qui est aussi la classe du cup produit
de la classe $a\in k^*/k^{*2}=H^1(k,\Z/2)$ et de  la classe $b\in k^*/k^{*2}=H^1(k,\Z/2)$.

Sur un corps $k$ de car. diff\'erente de 2, pour tout entier $n \geq 1$, et pour
$a_{1}, \dots, a_{n} \in k^*$,  la
$n$-forme de Pfister $<<a_{1}, \dots, a_{n}>>$ est la forme  quadratique 
en $2^n$ variables
d\'efinie par $<1,-a_{1}> \otimes \dots \otimes  <1,-a_{n}>$. 
De telles formes ont la propri\'et\'e qu'elles sont hyperboliques d\`es qu'elles sont isotropes.
On appelle voisine de Pfister d'une $n$-forme de Pfister $\psi$ une sous-forme $\phi$  de  $\psi$ de rang strictement plus grand que $2^n$. 
Les quadriques d\'efinies par une forme de Pfister et par un voisine de cette forme sont
stablement $k$-birationnellement \'equivalentes. C'est ainsi le cas de la conique d'\'equation 
$x^2-ay^2-bz^2=0$ et de la quadrique de $\P^3_{k}$ d'\'equation $x^2-ay^2-bz^2+abt^2=0$.

Une g\'en\'eralisation de la propri\'et\'e remarquable des coniques d\'ecrites ci-dessus
est le th\'eor\`eme suivant.

\begin{theo}\label{ovv}
Soit $k$ un corps de caract\'eristique diff\'erente de 2. Soit  $<<a_{1}, \dots, a_{n}>>$
une $n$-forme de Pfister.  Soit $Q\subset \P^{2^n-1}$ la quadrique lisse qu'elle d\'efinit.
 Le noyau de l'application naturelle
de groupes de cohomologie galoisienne
$$H^n(k,\Z/2) \to H^n(k(Q),\Z/2)$$
est engendr\'e par le cup-produit $(a_{1}) \cup \dots (a_{n})$, et il est non nul
si et seulement si la forme de Pfister est anisotrope, i.e. la quadrique $Q$ n'a pas de $k$-point.
\end{theo}

Ce th\'eor\`eme fut \'etabli pour $n=3$ en 1974 par Arason, avant les r\'esultats spectaculaires de Merkur'ev et Sousline.
Il avait \'et\'e pr\'ec\'ed\'e par un r\'esultat analogue d'Arason et Pfister pour les groupes de Witt, qu'on pourrait aussi utliser.
Il fut \'etabli pour $n=4$ par Jacob et Rost en 1989, et obtenu pour tout $n$ par Orlov, Vishik et Voevodsky \cite{OVV} en 2007 
comme cons\'equence des travaux de Voevodsky sur la conjecture de Milnor.

Une fois le point de vue birationnel adopt\'e dans \cite{ctoj}, il est devenu clair comment
\'etendre les r\'esultats de non rationalit\'e en dimension sup\'erieure.
Dans \cite{ctoj},  avec Ojanguren, nous construisons des vari\'et\'es a priori singuli\`eres $Y$
munies d'une fibration sur  $\P^3_{\C}$ dont la fibre g\'en\'erique est d\'efinie par
une  voisine d'une $3$-forme de Pfister anisotrope  $<<a_{1}, a_{2}, b_{3}c_{3}>>$ sur le corps $\C(\P^3)$, 
telle que la classe    $\beta= (a_{1}, a_{2}, b_{3})  \in H^3(\C(\P^3),\Z/2)$ soit
non nulle, car ramifi\'ee sur $\P^3_{\C}$, diff\'erente de $\alpha =   (a_{1}, a_{2}, a_{3}) \in H^3(\C(\P^3),\Z/2)$, car les ramifications 
sur $\P^3_{\C}$ diff\`erent,
et dont l'image $\beta_{\C(X)}$ est dans $H^3_{nr}(\C(X)/\C, \Z/2)$ car la ramification de $\beta$
est ``mang\'ee'' par celle de $\alpha$. Comme on a $$\beta \notin \{0,\alpha\}  \subset H^3(\C(\P^3), \Z/2),$$ le th\'eor\`eme 
\ref{ovv}, dans le cas $n=3$ (Arason) assure alors $\beta_{\C(X)}\neq 0$.

Pour  accomplir le programme, il faut trouver les \'el\'ements $a_{1},a_{2}, b_{3},c_{3} \in \C(\P^3)$.
On les obtient dans \cite{ctoj} comme des produits d'un nombre assez grand de formes lin\'eaires.

Dans \cite{Schr3}, Schreieder a r\'eussi \`a faire des constructions analogues sur $\P^n_{\C}$
pour tout $n$ (les $a_{i}$, $b_{j}$, $c_{j}$ faisant ici intervenir des formes de degr\'e 2 sur $\P^n$).
Le th\'eor\`eme \ref{ovv} donne alors des vari\'et\'es $X$ munies d'une fibration
sur $\P^n_{\C}$ dont la fibre g\'en\'erique est une (voisine d'une) $n$-quadrique de  Pfister
et qui satisfont $H^n_{nr}(\C(X),\Z/2)\neq 0$, et qui donc ne sont pas r\'etractilement rationnelles.

On trouve d'autres utilisations de ces id\'ees dans des travaux d'E. Peyre et de A. Asok.

\begin{rema}
Soit $k$ un corps. Soit $Q \subset \P^n_{k}$, $n \geq 2$ une quadrique lisse.
L'application
$\Br(k) \to  \Br(Q) = \Br_{nr}(k(Q)/k)$ est surjective.
Pour $i\geq 3$,  et ${\rm car}(k) \neq 2$,
le conoyau de 
$$H^{i}(k, \Q_{2}/\Z_{2}(i-1)) \to H^{i}_{nr}(k(Q)/k,\Q_{2}/\Z_{2}(i-1))$$
a \'et\'e \'etudi\'e par Kahn, Rost, Sujatha. 
Pour $i=3$, ils ont montr\'e que l'application est surjective, sauf peut-\^{e}tre
si $Q$ est d\'efinie par une forme d'Albert $<-a,-b,ab,c,d,-cd>$.
 
 \end{rema}
\subsection{Diff\'erentielles}

L'\'enonc\'e suivant est \'etabli par Totaro dans \cite{T}.

\begin{prop}\label{tot}
Soit $X$ une $k$-vari\'et\'e projective et lisse connexe sur un corps $k$.
Si  $X$  est $CH_{0}$-triviale, alors 
 $H^0(X,\Omega^{i})=0$ pour tout
 entier $i>0$.
\end{prop}

\begin{rema}
La d\'emonstration utilise des applications cycles \`a valeurs dans diverses th\'eories
cohomologiques, et des arguments de correspondances.
Si $X$ projective et lisse est presque ${\rm R}$-triviale, alors elle est $CH_{0}$-triviale (Lemme \ref{RtriviCH_{0}triv}),
et donc $H^0(X,\Omega^{i})=0$ pour tout $i>0$. 
\end{rema}

 \begin{prop}\label{remgen}
 Soit $k$ un corps. Soit $F$ un foncteur contravariant de la cat\'egorie des $k$-sch\'emas
 vers la cat\'egorie des ensembles.
 Supposons que pour toute $k$-vari\'et\'e lisse int\`egre $U$
 la fl\`eche $F(U) \to F(\P^1_{U})$ induite par la projection $\P^1_{U} \to U$ soit un isomorphisme.
 Soit $X$ une $k$-vari\'et\'e propre, int\`egre, g\'en\'eriquement lisse.
 Si $X$ est presque ${\rm R}$-triviale,
 alors il existe un ouvert  lisse
  non vide $U \subset X$ tel que 
 $$ Im(F(X) \to F(U)) = Im(F(k) \to F(U)),$$
 la fl\`eche $F(k) \to F(U)$ \'etant donn\'ee par la projection $U \to \Spec(k)$.
  \end{prop}
  \begin{proof}
  Soit $U \subset X$ un ouvert, et soit $g : P^1\times_{k}U \to X$  un $k$-morphisme.
  Soient $f_{1}, f_{2}$ deux sections de la projection $p: P^1\times_{k}U \to U$.
  Alors les applications $F(X) \to F(U)$ d\'efinies par
   $(g\circ f_{1})^*$ et $(g\circ f_{2})^*$ co\"{\i}ncident. En effet pour tout $\alpha \in F(X)$,
   on a $g^*(\alpha)=p^*(\beta)$, et donc  $f_{i}^* \circ g^*(\alpha)=  f_{i}^* \circ p^*(\beta)= \beta$
   pour $i=1,2$.
   
   Soit $F=k(X)$ le corps des fonctions de $X$,  soit $\eta \in X$ le point g\'en\'erique.
   Il existe  $n \in X(k)$ tel que
   sur $X_{F}$, les points $\eta \in X_{F}(F)$ et $n_{F}  \in X_{F}(F)$ sont ${\rm R}$-\'equivalents.
Comme $X$ est propre sur $k$, ceci implique qu'il existe un ouvert non vide $U \subset X$
et une famille finie de $k$-morphismes 
$f_{i} : \P^1\times U  \to X$, $i=0, \dots, s$, tels que $f_{0}(0,u) = u$, que $f_{s}(1,u)=n$,
et que $f_{i}(1,u)=f_{i+1}(0,u)$ pour $0\leq i < s$.
 L'\'enonc\'e r\'esulte alors de ce qui pr\'ec\`ede.
  \end{proof}
 
 \begin{prop}\label{Rtot}
 Soient $k$ un corps infini et $X$ une $k$-vari\'et\'e propre et lisse, g\'eom\'etriquement connexe.
 Si $X$ est presque ${\rm R}$-triviale, alors  $H^0(X,(\Omega^{i})^{\otimes m})=0$ pour tout $i>0$ et tout $m>0$.
\end{prop}
\begin{proof}  Pour toute $k$-vari\'et\'e lisse int\`egre $U$, tout entier $i>0$, tout entier $m>0$,
la fl\`eche de restriction
$$ H^0(U,(\Omega^{i})^{\otimes m}) \to H^0(\P^1_{U},(\Omega^{i})^{\otimes m})$$
est un isomorphisme, comme on voit en utilisant la formule donnant le faisceau des
diff\'erentielles sur un produit de $k$-vari\'et\'es,  et en utilisant le fait que sur la droite projective, on a
$\Omega^1_{\P^1}=O_{\P^1}(-2)$, et donc, pour tout $m>0$ ,toute section de $(\Omega^1_{\P^1})^{\otimes m}$ est nulle.
On applique alors la proposition pr\'ec\'edente au foncteur $U \mapsto H^0(U,(\Omega^{i})^{\otimes m}) $,
et on utilise le fait que l'application de restriction $H^0(X,(\Omega^{i})^{\otimes m})  \to H^0(U,(\Omega^{i})^{\otimes m})$
est injective.
\end{proof}

\subsection{Composantes connexes r\'eelles}

\begin{theo}
Soit $\R$ le corps des r\'eels. Soit $X$ une $\R$-vari\'et\'e projective, lisse,
g\'eom\'etriquement connexe, de dimension $d$. Soit $s\geq 0$ le nombre
de composantes connexes de $X(\R)$.

(a) L'entier $s$ est un invariant birationnel stable.

(b) Si $X$ est r\'etractilement rationnelle, alors $X(\R)$ est connexe.

(c1) Pour  $s\geq 1$,  on a $CH_{0}(X)/2 = (\Z/2)^s$.

(c2) Si deux points de $X(\R)$ sont rationnellement \'equivalents sur  $X$, alors ils
appartiennent \`a la m\^eme composante connexe de $X(\R)$.

(d1) Si $s=0$, pour tout entier $m \geq d+1$, on a $H^{m}_{nr}(\R(X)/\R, \Z/2)=0$.

(d2) Si $s \geq 1$, pour tout  entier $m \geq d+1$, on a $H^{m}_{nr}(\R(X)/\R, \Z/2) =(\Z/2)^s$.

(e) Si $X$ est g\'eom\'etriquement rationnellement connexe,  deux points de $X(\R)$ sont
${\rm R}$-\'equivalents si et seulement si ils sont dans la m\^eme composante connexe.
\end{theo}

\begin{proof}
Pour (a), il suffit de voir que si $U \subset X$ est un ferm\'e de codimension au moins 2,
alors $U(\R) \subset X(\R)$ induit une bijection sur les composantes connexes.
On utilise alors le fait qu'une $k$-application rationnelle d'une $k$-vari\'et\'e lisse
dans une $k$-vari\'et\'e propre est d\'efinie en dehors d'un ferm\'e de codimension au moins 2.
Sous l'hypoth\`ese de   (b), il existe un ouvert de Zariski $U \subset X$
tel que l'image de $U(\R)$ dans $X(\R)$ soit form\'e de points
directement R-li\'es sur $X$, donc dans la m\^{e}me composante connexe de $X(\R)$.
Comme pour la $\R$-vari\'et\'e lisse $X$ tout point de $X(\R)$ est limite de points de $U(\R)$,
ceci suffit \`a conclure que $X(\R)$ est connexe.
Pour (c), voir  CT-Ischebeck \cite{CTI}. Pour (d), voir  CT-Parimala \cite{ctpari}.
L'\'enonc\'e (e) fut \'etabli par Koll\'ar.
\end{proof}

En dimension $d=1$, tous ces \'enonc\'es remontent \`a Witt. 
C'est B. Segre \cite{segre} qui le premier remarqua que les surfaces cubiques lisses $X$ sur $\R$,
qui sont toutes $\R$-unirationnelles, ne sont pas $\R$-rationnelles si $X(\R)$ n'est pas 
connexe.

\section{Surfaces g\'eom\'etriquement rationnelles}

\begin{theo} (Enriques, Manin, Iskovskikh, Mori)
Soient $k$ un corps et $X$ une $k$-surface projective, lisse, g\'eom\'etriquement rationnelle.
Alors $X$ est $k$-birationnelle \`a une telle $k$-surface de l'un des deux types suivants :

(i) Surface de del Pezzo de degr\'e $d$, avec $1 \leq d \leq 9$.

(ii) Surface $X$ munie d'une fibration relativement minimale $X \to D$, o\`u $D$ est une conique
lisse, la fibre g\'en\'erique est une conique lisse, et toutes les fibres sont des coniques avec
au plus un point singulier.
\end{theo}

Rappelons que les surfaces de del Pezzo de degr\'e 3 sont les surfaces cubiques lisses.

C'est une question ouverte depuis longtemps si une  $k$-surface comme dans le th\'eor\`eme,
 d\`es qu'elle poss\`ede 
un $k$-point, est $k$-unirationnelle.  C'est connu pour les surfaces cubiques. 
Une r\'eponse affirmative impliquerait que  les vari\'et\'es  complexes de dimension 3 fibr\'ees en coniques sur le plan $\P^2_{\C}$
 sont unirationnelles, ce qui est une question ouverte encore plus connue.

Une question g\'en\'erale (Sansuc et l'auteur) sur les surfaces du type ci-dessus est : dans quelle mesure le module galoisien $\Pic(X^s)$ 
(qui est un groupe ab\'elien de type fini) et les objets
qui lui sont attach\'es contr\^olent la g\'eom\'etrie et l'arithm\'etique de $X$ ?

A-t-on la r\'eciproque du th\'eor\`eme \ref{picpermut} (c) :

  Question 1 CT-Sansuc 1977)  :  {\it Si $X(k) \neq \emptyset$ et  le module galoisien $\Pic(X^s)$  est un facteur direct d'un module de permutation,
  la $k$-vari\'et\'e $X$ est-elle facteur direct birationnel d'un espace projectif $\P^n_{k}$ ?}

La $K$-th\'eorie alg\'ebrique (id\'ees de S. Bloch, th\'eor\`eme de Merkurjev-Suslin) a permis d'\'etablir pour ces
surfaces, sans analyse cas par cas, la r\'eciproque du Th\'eor\`eme \ref{picpermut} (b).

\begin{theo} \cite{ctk2}
Soit $X$ une $k$-surface projective, lisse, g\'eom\'etriquement rationnelle, poss\'edant un z\'ero-cycle
de degr\'e 1. Si le module galoisien $\Pic(X^s)$ est un facteur direct d'un module de permutation,
alors $X$ est $CH_{0}$-triviale.
\end{theo}

\medskip

Voici quelques rappels de CT-Sansuc  \cite{CTSa87a}.
Soit $X$ une $k$-surface projective, lisse, g\'eom\'etriquement rationnelle. Soit 
$\Pic(X^s)$ le module galoisien d\'efini par le groupe de Picard. C'est le groupe des
caract\`eres d'un $k$-tore $S$. Pour tout $k$-tore $T$, on a une suite exacte de groupes ab\'eliens
$$ 0 \to H^1(k,T) \to H^1(X,T) \to Hom_{g}(\hat{T},\hat{S}) \to H^2(k,T) \to H^2(X,T),$$
o\`u la cohomologie est la cohomologie \'etale.
Si $X$ poss\`ede un $k$-point, la fl\`eche $H^2(k,T) \to H^2(X,T)$ a une r\'etraction, donc on a
une suite exacte
$$ 0 \to H^1(k,T) \to H^1(X,T) \to Hom_{g}(\hat{T},\hat{S}) \to 0.$$
On appelle torseur universel sur $X$ un torseur ${\mathcal{T}} \to X$ sous le $k$-tore $S$ dont la classe dans
$ H^1(X,S)$ a pour image l'identit\'e dans  $Hom_{g}(\hat{S},\hat{S}) $. Si $X$ poss\`ede un 
$k$-point $P \in X(k)$, il existe un torseur universel, et on peut le fixer
(\`a automorphisme de $S$-torseur pr\`es) en demandant que sa fibre en $P$ soit
triviale, ce qui \'equivaut au fait qu'il existe un $k$-point  de ${\mathcal{T}} $ d'image $P$
dans $X$.

Un torseur universel   ${\mathcal{T}}$ sur une $k$-surface projective, lisse, g\'eom\'etriquement rationnelle
est une $k$-vari\'et\'e g\'eom\'etriquement rationnelle (ouverte) de dimension $2 + rg(\Pic(X^s))$.

Question  2 (CT-Sansuc 1977). {\it Sur une $k$-surface projective et lisse g\'eom\'etri\-quement rationnelle $X$,
les torseurs universels $\mathcal{T}$ avec un $k$-point sont-ils $k$-rationnels ?}

Ceci a \'et\'e \'etabli pour les surfaces fibr\'ees en coniques au-dessus de $\P^1_{k}$
avec au plus 4 fibres g\'eom\'etriques non lisses. C'est d'ailleurs ce qui a men\'e aux exemples
de vari\'et\'es stablement rationnelles non rationnelles \cite{BCTSSD}.

En 1977,  Sansuc et moi avions \'etabli que si $Y_{c}$ est une compactification lisse
d'un $k$-torseur universel  $\mathcal{T}$  alors
 $\Pic({\overline{Y}}_{c})$ est un $g$-module de permutation,
 et  $\Br(Y_{c})/\Br(k)=0$.
 
 Pour tester l'\'eventuelle non rationalit\'e des torseurs universels sur
 les surfaces g\'eom\'etriquement rationnelles, on peut essayer de calculer
 les invariants cohomologiques sup\'erieurs $H^{i}_{nr}(k(\mathcal{T})/k, \Q/\Z(i-1))$.
 Dans sa  th\`ese, Yang CAO a \'etabli le th\'eor\`eme suivant, qui s'applique
 en particulier aux  $k$-surfaces cubiques lisses.
 \begin{theo}\cite{Cao}
 Soit  $X$ une surface projective, lisse connexe, g\'eom\'e\-tri\-que\-ment rationnelle
  sur un corps $k$. Si $X$ n'est pas $k$-birationnelle \`a une surface de del Pezzo
  $k$-minimale de degr\'e 1, et si   $\mathcal{T}$ est
 un torseur universel sur $X$ avec un $k$-point,
 $H^{3}_{nr}(k(\mathcal{T})/k, \Q/\Z(2))/H^3(k,\Q/\Z(2))$
 est un groupe de torsion $2$-primaire.
  \end{theo}

Soit $k$ un corps de car. diff\'erente de 2 poss\'edant une extension finie  $L=k[t]/P(t)$ de degr\'e 3,
de cl\^{o}ture galoisienne $K/k$ de groupe $S_{3}$, et soit $k(\sqrt{a})$ l'extension discriminant.
Dans \cite{BCTSSD}, on a montr\'e que la surface g\'eom\'etriquement rationnelle d'\'equation affine
$y^2-az^2=P(x)$ est stablement $k$-rationnelle mais non $k$-rationnelle.
Ceci fut utilis\'e dans \cite{BCTSSD}  pour donner des exemples de vari\'et\'es de dimension 3 sur $\C$
qui sont stablement rationnelles mais non rationnelles.

Hassett avait soulev\'e la question si de tels exemples existent sur un corps $k$ parfait
dont la cl\^oture alg\'ebrique est procyclique, par exemple sur un corps fini.

 Le th\'eor\`eme suivant, qu'on confrontera avec la question 1 ci-dessus, n'admet pour l'instant qu'une d\'emonstration
 extr\^emement calculatoire, passant par l'analyse (faite par plusieurs auteurs) de toutes les actions possibles
 du groupe de Galois absolu sur le groupe de Picard g\'eom\'etrique des surfaces
 de del Pezzo de degr\'e 3, 2, 1, ce qui implique des groupes de type $E_{6}, E_{7}, E_{8}$.

\begin{theo} \cite{ctcyclique}
Soient $k$ un corps et $X$ une $k$-surface projective, lisse, g\'eom\'etriquement rationnelle.
Supposons que $X$ poss\`ede un point $k$-rationnel et que
$X$ soit d\'eploy\'ee par une extension cyclique de $k$.
Si  $X$ n'est pas  $k$-rationnelle, alors  il existe une extension finie 
s\'eparable $k'/k$ telle que $\Br(X_{k'})/\Br(k') = H^1(k', \Pic(X^s))\neq 0$, et alors
$X$ n'est pas stablement $k$-rationnelle.
\end{theo}

Soit $X$  d\'eploy\'ee par une extension cyclique de $k$.
Si l'on suppose $\Br(X_{k'})/\Br(k')= H^1(k', \Pic(X^s))=0$ pour toute extension s\'eparable $k'$ de $k$,
on peut montrer (Endo-Miyata) que $\Pic(X^s)$ est un facteur direct d'un module de permutation.
Tout torseur universel $\mathcal{T}$ avec un $k$-point est alors 
$k$-birationnel \`a $X \times_{k} S$. Si de tels torseurs universels
\'etaient automatiquement $k$-rationnels (question 2 ci-dessus), l'hypoth\`ese $\Br(X_{k'})/\Br(k')=0$
pour tout $k'/k$ fini
  impliquerait que $X$ est facteur direct d'une $k$-vari\'et\'e $k$-rationnelle.

\section{Hypersurfaces cubiques}

\subsection{Rationalit\'e, unirationalit\'e, $CH_{0}$-trivialit\'e}

Soit $X \subset \P^n_{k}$ avec $n \geq 3$ une hypersurface cubique lisse avec $X(k)\neq \emptyset$.
On sait que $X$ est $k$-unirationnelle (..., B. Segre, ..., J. Koll\'ar).
Si $X$ contient une $k$-droite, alors $X$ est $k$-unirationnelle de degr\'e 2.
Je renvoie \`a \cite{ACTP} pour plus de rappels et des r\'ef\'erences \`a la litt\'erature.

Pour $n=2m+1$ impair quelconque, il existe des hypersurfaces cubiques lisses $X \subset \P^n_{k}$
qui sont $k$-rationnelles. C'est le cas de celles qui contiennent un ensemble globalement
$k$-rationnel d'espaces lin\'eaires $\Pi_{1}$, $\Pi_{2}$, chacun d\'efini sur une extension au plus
quadratique s\'eparable de $k$, et sans point commun.
Il en est ainsi de l'hypersurface cubique de Fermat $X_{n}$, d\'efinie par l'\'equation.
$$\sum_{i=0}^n x_{i}^{3}=0.$$
Elle poss\`ede une paire globalement $k$-rationnelle de sous-espaces lin\'eaires de dimension $m$
 gauches l'un \`a l'autre, \`a savoir
 $$x_{0}+jx_{1}= x_{2}+jx_{3}= \dots = x_{2m}+j x_{2m+1}=0$$
 et son conjugu\'e ($j$ est une racine primitive cubique de $1$).

Pour tout entier $n\geq 3$, l'hypersurface cubique de Fermat $X_{n}$ contient
  une $k$-droite, et donc est $k$-unirationnelle de degr\'e 2.

\medskip

Pour simplifier, supposons dans la suite de ce paragraphe $k=\C$,
et consid\'erons des hypersurfaces cubiques lisses $X \subset \P^n_{\C}$, $n\geq 3$.

Toute hypersurface cubique  $X \subset \P^n_{\C}$, $n \geq 3$
contient une droite, et est donc  est unirationelle de degr\'e 2.

Si  une hypersurface cubique est aussi unirationelle de degr\'e impair,
alors elle est $CH_{0}$-triviale et tous les invariants de type cohomologie non ramifi\'ee
sont universellement triviaux. On ne sait pas si $X$ est alors r\'etractilement rationnelle.

\medskip

Un th\'eor\`eme fameux de Clemens et Griffiths dit qu'aucune $X$ dans $\P^4_{\C}$
n'est rationnelle. Pour $n=2m$ pair quelconque on ne conna\^{i}t aucune 
$X$ dans $\P^{2m}_{\C}$ qui soit rationnelle, ou m\^{e}me r\'etractilement rationnelle.
Mais par ailleurs on n'en conna\^{\i}t aucune dont on sache qu'elle n'est pas
r\'etractilement rationnelle.
 .

C. Voisin a montr\'e que sur une union d\'enombrable de ferm\'es de codimension 3
de leur espace de modules, les hypersurfaces cubiques $X \subset \P^4_{\C}$ correspondantes 
sont $CH_{0}$-triviales.

\medskip

 Hassett et Tschinkel ont d\'ecrit des classes d'hypersurfaces cubiques de $\P^5_{\C}$
 qui sont unirationelles de degr\'e impair.   
Dans \cite{ctJAG} on en donne     dans $\P^n_{\C}$ pour
 tout $n$ de la forme $6m-1, 6m+1, 6m+3$. Elles sont ``presque diagonales''
 mais plus g\'en\'erales que l'hypersurface de Fermat.

 Hassett, et d'autres, ont d\'ecrit des sous-vari\'et\'es de l'espace de modules des hypersurfaces cubiques
 de $\P^5_{\C}$ dont les hypersurfaces correspondantes sont rationnelles (outre celles contenant
 deux plans gauches congugu\'es). Ces sous-vari\'et\'es sont contenues dans  une union d\'enombrable
 de diviseurs ``sp\'eciaux'' de l'espace de modules.
 
  C. Voisin \cite{voisinCub} a montr\'e que sur beaucoup  de ces diviseurs sp\'eciaux, les  hypersurfaces cubiques
 correspondantes  de $\P^5_{\C}$ sont $CH_{0}$-triviales.
 
Soit $X$ une $k$-vari\'et\'e projective, lisse, connexe. Suivant \cite{voisinCub}, on d\'efinit la $CH_{0}$-dimension essentielle $\delta(X)$
 de $X$ comme la plus petite dimension d'une $\C$-vari\'et\'e $Y$ projective, lisse, connexe munie d'un morphisme $Y  \to X$
 tel que pout tout corps $F$ contenant $\C$, l'application induite $CH_{0}(Y_{F}) \to CH_{0}(X_{F})$ soit surjective.
 
  C. Voisin \cite{voisinCub}  a montr\'e que pour les hypersurfaces cubiques lisses $X \subset \P^n_{\C} $
{\it tr\`es g\'en\'erales} avec $n=5$ ou $n \geq 4 $ pair, 
si $\delta(X) < {\rm dim}(X)$, alors $\delta(X)=0$, i.e. $X$ est $CH_{0}$-triviale.

\subsection{Hypersurfaces cubiques  presque diagonales}

Dans \cite{ctJAG} je donne en toute dimension des classes   explicites d'hypersurfaces cubiques lisses
complexes qui sont $CH_{0}$-triviales.   
Certains des r\'esultats valent sur un corps non alg\'ebriquement clos, comme on va le voir.
 Voici une variation sur la proposition 3.5 de l'article \cite{ctJAG}.

\bpr \label{H1nul}
Soient $k$ un corps  et  $X$ une $k$-vari\'et\'e projective et lisse telle que $H^1(X,O_{X})=0$, poss\'edant un $k$-point.
S'il existe une courbe $\Gamma/k$ projective,  lisse, connexe, avec un $k$-point,  et un $k$-morphisme $\Gamma \to X$
tels que, pour tout corps $F$, l'application induite
$CH_{0}(\Gamma_{F}) \to CH_{0}(X_{F})$ soit surjective, alors, pour tout corps $F$, l'application
 ${\rm deg}_{F} :CH_{0}(X_{F}) \to \Z$ est un isomorphisme, en d'autres termes la $k$-vari\'et\'e $X$ est 
$CH_{0}$-triviale.
\epr

 \begin{proof}  
 Soit $J$ la jacobienne de $\Gamma$. Pour tout corps $F$, on a
  $A_{0}(\Gamma_{F})=J(F)$.
Notons $K=k(X)$ le corps des fonctions de $X$.
L'hypoth\`ese $H^1(X,O_{X})=0$
implique que
 la vari\'et\'e d'Albanese de 
$X$ est triviale.
Un point de $J(k(X))$ d\'efinit une $k$-application rationnelle
de $X$ dans $J$,  donc un $k$-morphisme de $X$ dans $J$ car une application rationnelle d'une vari\'et\'e lisse dans une vari\'et\'e ab\'elienne est partout d\'efinie.
Mais comme la vari\'et\'e d'Albanese de $X$ est triviale, tout tel morphisme est constant. On a donc $J(k)=J(k(X))$.
Ainsi l'image de $A_{0}(\Gamma_{K})$ dans $A_{0}(X_{K})$
est dans l'image de l'application compos\'ee
$$J(k)= A_{0}(\Gamma) \to A_{0}(X)  \to A_{0}(X_{K}).$$
Par hypoth\`ese, l'application  $A_{0}(\Gamma_{K}) \to A_{0}(X_{K})$ est surjective. Ainsi  la restriction $CH_{0}(X) \to CH_{0}(X_{K})$ est surjective, et
en particulier la classe du point g\'en\'erique $\eta $ de $X$, 
qui d\'efinit un point de $X(K)$, a une classe dans $CH_{0}(X_{K})$ qui est dans l'image de $CH_{0}(X)$. D'apr\`es la proposition \ref{merku} (Merkurjev), ceci assure que la $k$-vari\'et\'e $X$ est  
$CH_{0}$-triviale.
 \end{proof}
 
 \begin{rema}
 Soit $k=\C$. 
  L'\'enonc\'e ci-dessus implique que si $CH_{0}(X)=\Z$,  alors $\delta(X) \leq 1$ implique $\delta(X)=0$.
R. Mboro \cite{mboro}  a \'etabli l'\'enonc\'e suivant. 
Supposons $CH_{0}(X)=\Z$, $H^2_{Betti}(X,\Z)_{tors}=0$ et $H^3_{Betti}(X,\Z)=0$.
Alors $\delta(X) \leq 2$ implique $\delta(X)=0$.
 \end{rema}

 \begin{theo}
 Soit $k$ un corps infini, de caract\'eristique diff\'erente de 3. Soit $f(x,y,z) \in k[x,y,z]$, resp. $g(u,v) \in k[u,v]$ des
 formes cubiques non singuli\`eres.  Soit $X \subset \P^4_{k}$  l'hypersurface cubique lisse d'\'equation $$f(x,y,z)-g(u,v)=0.$$
 
 Faisons les hypoth\`eses suivantes, qui sont   satisfaites si $k$ est un corps alg\'ebriquement clos :
 
 (a)   Il existe $a\in k^*$ tel que la surface cubique $S$ de $\P^3_{k}$ d'\'equation 
 $f(x,y,z)-at^3=0$ soit une surface $k$-rationnelle et que la courbe $\Gamma$ de $\P^2_{k}$ d'\'equation $g(u,v)-at^3=0$ poss\`ede
 un $k$-point.
 
 (b) L'hypersurface $X$ contient une $k$-droite.
 
\noindent  Alors l'hypersurface $X$ est $CH_{0}$-triviale. 
  \end{theo}

\begin{proof}  On note $A_{0}(X) \subset CH_{0}(X)$ le sous-groupe des classes de z\'ero-cycles de degr\'e z\'ero.
L'hypoth\`ese (b) implique que $X$ est $k$-unirationnelle de degr\'e 2, ce qui implique
$2 A_{0}(X_{F})=0$ pour tout corps $F$ contenant $k$. 
On a une application rationnelle dominante, de degr\'e 3, de $S \times \Gamma$ vers $X$,
qui envoie le produit des  
  vari\'et\'es affines $f(x,y,z)-a=0$ et $g(u,v)-a=0$ vers le point de coordonn\'ees homog\`enes
   $(x,y,z,u,v) \in X \subset \P^4_{k}$     (on suppose ici $k$ infini).
    En utilisant le fait que $S$ est $k$-rationnelle, on 
   montre que pour tout corps $F$ contenant $k$, il existe un
   $k$-morphisme $f: \Gamma \to X$ tel que pour tout corps $F$ contenant $k$,
   on ait $3 CH_{0}(X_{F}) \subset f_{*}(CH_{0}(\Gamma_{F})$.
Comme on a $2 A_{0}(X_{F})=0$, on en d\'eduit $CH_{0}(X_{F}) \subset f_{*}(CH_{0}(\Gamma_{F}))$.
La proposition \ref{H1nul} donne alors que $X$ est $CH_{0}$-triviale. 
\end{proof}

 L'\'enonc\'e ci-dessus se g\'en\'eralise en dimension sup\'erieure \cite[Prop. 3.7]{ctJAG}.
 Les arguments de \cite[Prop. 3.7 (i)]{ctJAG} et la proposition ci-dessus permettent 
  d'\'etablir que, sur tout corps $k$ de  diff\'erente de 3,
pour tout entier $n \geq 3$, impair ou non, l'hypersurface cubique de Fermat $X \subset \P^n_{\Q}$, $n \geq 3$, est $CH_{0}$-triviale.   
Pour tout $n=2m \geq 4$, c'est ainsi une question ouverte si cette  hypersurface est r\'etractilement rationnelle, 
ou m\^{e}me stablement rationnelle sur le corps $\Q$.  

\medskip

Sur le corps $k=\C$, la m\'ethode ci-dessus 
et des \'enonc\'es d'unirationalit\'e plus ou moins classiques permettent d'\'etablir 
l'\'enonc\'e g\'en\'eral suivant.

\begin{theo}\label{Hauptsatz2} \cite[Thm. 3.8]{ctJAG}
Toute hypersurface cubique lisse $X \subset \P^n_{\C}$ de dimension au moins 2
 dont l'\'equation
est donn\'ee par une forme   $\sum_{i}  
\Phi_{i}$, o\`u les $\Phi_{i}$ sont \`a variables
 s\'epar\'ees et chacune  a au plus 3 variables,
est   $CH_{0}$-triviale.
\end{theo}

\section{Sp\'ecialisation}\label{specialR}

\subsection{Sp\'ecialisation de la ${\rm R}$-\'equivalence et de l'\'equivalence rationnelle sur les z\'ero-cycles}

L'\'enonc\'e suivant est ``bien connu''. Pour une d\'emonstration 
d\'etaill\'ee  pour $\X/R$ projectif, on consultera la note de D. Madore \cite{madore}.
Voir aussi \cite{KS}.

\begin{theo}\label{specialisationR}
Soit $R$ un anneau de valuation discr\`ete excellent, $K$ son corps des fractions, $k$ son
corps r\'esiduel. Soit $\X$ un $R$-sch\'ema propre, $X=\X\times_{R}K$ la fibre g\'en\'erique
et $Y=\X\times_{R}k$ la fibre sp\'eciale.
L'application de r\'eduction  $X(K) = \X(R)  \to Y(k)$ induit une application
$X(K)/{\rm R} \to Y(k)/{\rm R}$.
\end{theo}
\begin{proof}  (Esquisse) Soit $\P^1_{K} \to X$ un $k$-morphisme.
Il s'\'etend en une application rationnelle de $\P^1_{R} $ vers $ \X$.
Par \'eclatements successifs de points ferm\'es sur $\P^1_{R}$,
 on obtient $Z \to \P^1_{R}$
et un $R$-morphisme $Z \to \X$ \'etendant l'application rationnelle.
La fibre $Z_{k}$ est g\'eom\'etriquement un arbre, dont les composantes sont des
droites projectives. Comme on voit par r\'ecurrence sur le nombre d'\'eclatements,
la r\'eunion $T$ des composantes de $Z_{k}$ obtenues par \'eclatement de $k$-points
forme elle-m\^eme un arbre form\'e de droites projectives $\P^1_{k}$,
 dont les intersections deux \`a deux sont
\'egales \`a un unique $k$-point, et tout $k$-point de $Z_{k}$
est contenu dans $T$. 
Les points $0$ et $\infty$ de $\P^1(K)=Z(K)$ s'\'etendent en des
sections $s_{0}$ et $s_{\infty}$ de $Z \to \Spec(R)$. Les sp\'ecialisations de
ces sections au-dessus de $\Spec(k)$ sont des $k$-points de $Z_{k}$,
qui sont dans le sous-arbre $T$.  Les images de $0_{K}$ et $\infty_{K}$
dans $Y(k)$ sont donc des points ${\rm R}$-\'equivalents sur $Y$.
\end{proof}

Soit $R$ un anneau de valuation discr\`ete excellent.
Soit $\X$ un $R$-sch\'ema projectif et plat, $\X_{K}$ la fibre g\'en\'erique et $\X_{k}$
la fibre sp\'eciale.

\'Etant donn\'e un point ferm\'e $P \in \X_{K}$, notons $\tilde{P}$ son adh\'erence
dans $\X$. C'est un $R$-sch\'ema fini.
On a une immersion ferm\'ee $\tilde{P}\times_{R} \Spec(k) \hookrightarrow \X_{k}$.
On associe \`a ce $R$-sch\'ema fini une combinaison lin\'eaire \`a coefficients entiers
de points ferm\'es de $\X_{k}$. Les coefficients sont d\'efinis par les longueurs \'evidentes.
Le z\'ero-cycle obtenu sur $\X_{k}$ peut aussi \^etre vu comme le z\'ero-cycle
associ\'e au $k$-sch\'ema  d\'ecoup\'e par $\pi=0$ sur  $\tilde{P}$.

Ceci d\'efinit une application lin\'eaire $Z_{0}(\X_{K}) \to Z_{0}(\X_{k})$.
On v\'erifie que ce processus est fonctoriel covariant en les morphismes (propres) de
 $R$-sch\'emas projectifs et plats.

Le th\'eor\`eme suivant est un cas particulier d'un th\'eor\`eme de Fulton pour  les groupes
de Chow de cycles de dimension quelconque.

\begin{theo} (Fulton)
Soit $R$ un anneau de valuation discr\`ete excellent, $K$ son corps des fractions, $k$ son
corps r\'esiduel, $\pi$ une uniformisante. Soit $\X$ un $R$-sch\'ema projectif et plat, 
$X=\X\times_{R}K$ la fibre g\'en\'erique
et $Y=\X\times_{R}k$ la fibre sp\'eciale.  
Il existe un  unique homomorphisme
de sp\'ecialisation 
$$CH_{0}(X) \to CH_{0}(Y)$$
qui  associe \`a la classe d'un point ferm\'e $P$ de $X$ d'adh\'erence $\tilde{P} \subset \X$
la classe du z\'ero-cycle associ\'e au diviseur de Cartier d\'ecoup\'e par $\pi=0$ sur  $\tilde{P}$.
\end{theo}

C'est \'enonc\'e au d\'ebut du \S 20.3 de \cite{fulton}, avec r\'ef\'erence au \S 6.2 et au th\'eor\`eme 6.3.
On part d'une suite exacte facile
$$ CH_{1}(Y) \to CH_{1}(\X/R) \to CH_{0}(X) \to 0$$
\'etablie au \S 1.8. 

On utilise ensuite un 
homomorphisme de Gysin $i^{!} : CH_{1}(\X/R) \to  CH_{0}(Y)$
  introduit au \S 6.2. Il est d\'emontr\'e au  Th\'eor\`eme 6.3 que le compos\'e
  $CH_{1}(Y) \to CH_{1}(\X/R) \to CH_{0}(Y)$ est nul, en utilisant le fait
  que $Y$ est un diviseur de Cartier principal sur $\X$.
  Ceci induit un homomorphisme de sp\'ecialisation $ CH_{0}(X) \to  CH_{0}(Y)$.
  
 Autant que je puisse voir, le \S 2, et la Proposition 2.6 de \cite{fulton}, qui utilisent
 un homomorphisme de Gysin $ i^{*} : CH_{1}(\X/R) \to  CH_{0}(Y)$,
  suffisent 
 pour \'etablir ces r\'esultats. Ils reposent sur un th\'eor\`eme fondamental, le Th\'eor\`eme 2.4. 
La D\'efinition 2.3 of \cite{fulton} donne pr\'ecis\'ement la description de l'homomorphisme
de sp\'ecialisation donn\'ee dans l'\'enonc\'e ci-dessus.

\medskip

\begin{rema}
On peut facilement ramener la d\'emonstration de l'\'enonc\'e ci-desssus
 au cas o\`u  $\X$ est une $R$-courbe plate, projective, connexe, r\'eguli\`ere.
Mais ce cas-l\`a ne semble pas plus facile que le cas g\'en\'eral, si la $R$-courbe
n'est pas lisse.
Or c'est tout le point : si la fibre sp\'eciale est une union de diviseurs lisses $Y_{i}/k$ (non
principaux),
on n'a pas en g\'en\'eral de fl\`eches $CH_{0}(X) \to CH_{0}(Y_{i})$ qui par somme donneraient
la fl\`eche $CH_{0}(X) \to CH_{0}(Y)$. Par ailleurs, si $Y/k$ n'est pas lisse, la fl\`eche naturelle
$\Pic(Y) \to CH_{0}(Y)$ n'est a priori ni injective ni surjective.
\end{rema}
 
\subsection{Non rationalit\'e  stable par sp\'ecialisation singuli\`ere}

 Les deux th\'eor\`emes suivants, qui g\'en\'eralisent un argument de C. Voisin \cite{voisinInv},
  sont \'etablis 
 dans \cite{CTP16} en utilisant la sp\'ecialisation de Fulton des groupes de Chow (des z\'ero-cycles).
 Ils ont d\'ej\`a  \'et\'e discut\'es dans divers textes, en particulier dans  \cite{peyre} et  \cite{P}.
  On d\'eveloppe ici la remarque 1.19 de \cite{CTP16} :
on donne une d\'emonstration qui utilise la sp\'ecialisation de la ${\rm R}$-\'equivalence,
 plus simple \`a \'etablir que celle du groupe de Chow des z\'ero-cycles.

\begin{theo}  \cite[Thm. 12]{CTP16}.
Soit $A$ un anneau de valuation discr\`ete, $K$ son corps des fractions, suppos\'e de
caract\'eristique z\'ero, $k$ son
corps r\'esiduel.
 Soit $\X$ un $A$-sch\'ema projectif et plat, 
$X=\X\times_{A}K$ la fibre g\'en\'erique
et $Y=\X\times_{A}k$ la fibre sp\'eciale. Supposons 
$X/K$ lisse et  g\'eom\'etriquement int\`egre et
$Y/k$ g\'eom\'etriquement int\`egre.  
Supposons  $Y(k)$ Zariski dense dans $Y$ et
 qu'il existe une r\'esolution des singularit\'es projective 
$f : Z \to Y$ qui est   un $CH_{0}$-isomorphisme.
Si la $K$-vari\'et\'e $X$ est r\'etractilement rationnelle, 
alors 
la $k$-vari\'et\'e $Z$ est  $CH_{0}$-triviale.
\end{theo}
\begin{proof} 
On proc\`ede au d\'ebut comme dans \cite[Thm. 12]{CTP16}.
On note $B$ le compl\'et\'e de l'anneau local de $\X$ au point g\'en\'erique  $\eta$ de $Y$.
Soit $F$ son corps des fractions.
La fl\`eche $A \to B$ est un homomorphisme local, induisant $k \to k(Y)$ sur les corps
r\'esiduels. 
On consid\`ere le $B$-sch\'ema $\X\times_{A}B$. Sa fibre sp\'eciale
est $Y\times_{k} k(Y)$, qui admet la d\'esingularisation $Z\times_{k}k(Y) \to Y\times_{k} k(Y)$.
Le $k(Y)$-morphisme $Z\times_{k}k(Y) \to Y\times_{k} k(Y)$ est  
$CH_{0}$-trivial. Soit $U \subset Y_{lisse}$ un ouvert tel que $f^{-1}(U) \to U$ soit un isomorphisme.
Soit $P \in U(k)$. Soit $M \in Z(k)$ son image r\'eciproque sur $f^{-1}(U)$.
Par Hensel, le point g\'en\'erique $\eta \in Y(k(Y))$ et le point $P_{k(Y)}$ se rel\`event en des 
$F$-points de $\X\times_{K}F$.  Comme $X$ est  lisse et r\'etractilement rationnel sur le corps
$K$ qui est de caract\'eristique z\'ero, ces deux points  
sont ${\rm R}$-\'equivalents sur
 $\X\times_{K}F$. Par sp\'ecialisation de la ${\rm R}$-\'equivalence,  les points $\eta$ et $P_{k(Y)}$
 sont ${\rm R}$-\'equivalents sur $Y_{k(Y)}$. Ils sont donc rationnellement \'equivalents sur 
 $Y_{k(Y)}$. Soit $\xi$ le point g\'en\'erique de $Z$ d'image $\eta \in Y$.
 L'hypoth\`ese que $f$ est un $CH_{0}$-isomorphisme implique 
 que  
   $\xi_{k(Z)}$ est rationnellement \'equivalent \`a $M_{k(Z)}$ sur $Z_{k(Z)}$.
 Ceci implique que la $k$-vari\'et\'e projective et lisse $Z$ est 
 $CH_{0}$-triviale (\cite[Lemma 1.3]{ACTP}, \cite[Prop. 1.4]{CTP16}).
\end{proof}

\begin{prop}\cite[Prop. 1.8]{CTP16}
Soit $f :Z \to Y$ une r\'esolution des singularit\'es.
Pour \'etablir  que  sur tout corps $F$ contenant $k$, l'homomorphisme
$f_{*} : CH_{0}(Z_{F}) \to CH_{0}(Y_{F})$ est un isomorphisme, il suffit
de montrer que pour tout point $M$ du sch\'ema $Y$,
le $k(M)$-sch\'ema fibre $Z_{M}$ est $CH_{0}$-trivial. $\Box$
 \end{prop}

\medskip

\begin{theo}  \cite[Thm. 14]{CTP16}.
Soit $A$ un anneau de valuation discr\`ete, $K$ son corps des fractions,
 $k$ son
corps r\'esiduel suppos\'e alg\'ebriquement clos.
Soit $\overline{K}$ une cl\^oture alg\'ebrique de $K$.
 Soit $\X$ un $A$-sch\'ema projectif et plat, 
$X=\X\times_{A}K$ la fibre g\'en\'erique
et $Y=\X\times_{A}k$ la fibre sp\'eciale. Supposons 
$X/K$ lisse et  g\'eom\'etriquement int\`egre et
$Y/k$ g\'eom\'etriquement int\`egre.  
Supposons   
 qu'il existe une r\'esolution des singularit\'es projective 
$f : Z \to Y$ qui est   un $CH_{0}$-isomorphisme.
Si la $\overline{K}$-vari\'et\'e $X\times_{K}\overline{K}$ est r\'etractilement rationnelle, 
alors 
la $k$-vari\'et\'e $Z$ est   $CH_{0}$-triviale.
\end{theo}

\begin{proof}
Comme dans \cite{CTP16}, ceci se d\'eduit du th\'eor\`eme pr\'ec\'edent par
une r\'eduction simple.
\end{proof}

Dans \cite[\S 2.4]{P},   A. Pirutka d\'eveloppe une autre variante
de la remarque 1.19 de \cite{CTP16}.

 \begin{theo} 
Soit $A$ un anneau de valuation discr\`ete, de corps des fractions $K$
et de corps r\'esiduel $k$.  
 Soit $\X$ un $A$-sch\'ema projectif et plat, 
$X=\X\times_{A}K$ la fibre g\'en\'erique
et $Y=\X\times_{A}k$ la fibre sp\'eciale. Supposons $X/K$  et
$Y/k$ g\'eom\'etriquement int\`egres, et $Y(k)$ Zariski dense dans $Y$.  
Supposons   
 qu'il existe une r\'esolution des singularit\'es projective 
$f : Z \to Y$ qui soit ${\rm R}$-triviale.
Si  $X$ est   r\'etractilement rationnelle, alors 
$Z$ est presque ${\rm R}$-triviale.
En particulier, il existe un point $M \in Z(k)$ tel que
 le point g\'en\'erique de $Z$ est, sur $Z_{k(Z)}$, ${\rm R}$-\'equivalent \`a
$M_{k(Z)}$. 
\end{theo}

 Cet \'enonc\'e implique  le suivant.

\begin{theo}\label{specRetdiff}
Soit $A$ un anneau de valuation discr\`ete, de corps des fractions $K$
et de corps r\'esiduel $k$ alg\'ebriquement clos. 
 Soit $\X$ un $A$-sch\'ema projectif et plat, 
$X=\X\times_{A}K$ la fibre g\'en\'erique
et $Y=\X\times_{A}k$ la fibre sp\'eciale. Supposons $X/K$  et
$Y/k$ g\'eom\'etriquement int\`egres.  
Supposons   
 qu'il existe une r\'esolution des singularit\'es projective 
$f : Z \to Y$ qui soit ${\rm R}$-triviale.
Si  $X$ est g\'eom\'etriquement  r\'etractilement rationnelle, alors 
$Z$ est presque ${\rm R}$-triviale.
En particulier :

(i) Il  existe un point $M \in Z(k)$ tel que
 le point g\'en\'erique de $Z$ est, sur $Z_{k(Z)}$, ${\rm R}$-\'equivalent \`a
$M_{k(Z)}$.

(ii) Pour tous entiers $i>0$ et $m>0$, on a
$$H^0(Z,(\Omega^{i})^{\otimes m})=0.$$
\end{theo}

\begin{proof}
Voir \cite[Thm. 2.14]{P}.
Pour la derni\`ere assertion, voir la proposition  \ref{Rtot} ci-dessus.
\end{proof}

 Pour \'etablir dans des cas concrets que la r\'esolution
 $f : Z \to Y$ est ${\rm R}$-triviale, on utilise l'\'enonc\'e suivant.

\begin{prop} Soit $f: Z  \to Y$ un $k$-morphisme propre.
 Si pour tout corps $F$ contenant $k$ et
  tout point $M \in Y(F)$, la $F$-vari\'et\'e fibre $Z_{M}$ 
est ${\rm R}$-triviale, alors
$f$ est ${\rm R}$-triviale. $\Box$
\end{prop}

La d\'emonstration de cet \'enonc\'e est facile, mais \'etablir
que l'hypoth\`ese sur les fibres $Z_{M}$ vaut est l'une des
principales difficult\'es dans la pratique.

\subsection{Applications aux  vari\'et\'es alg\'ebriques complexes}

Elles sont nombreuses. Certaines ont \'et\'e  d\'ecrites dans les  rapports
\cite{peyre}, \cite{P}. 

Pour des familles projectives et lisses $\X \to S$ de
 vari\'et\'es alg\'ebriques d'un ``type donn\'e'', param\'etr\'ees par
une vari\'et\'e alg\'ebrique complexe, on \'etablit des th\'eor\`emes
du type :

L'ensemble des points   $s \in S(\C)$ tels que la fibre $\X_{s}$ ne soit 
pas r\'etractilement rationnel est Zariski dense dans $S$.

On  montre en fait que l'ensemble des points $s$ o\`u $X_{s}$
est r\'etractilement rationnel est contenu dans
une union d\'enombrable de
ferm\'es propres de $S$.

On s'int\'eresse bien s\^{u}r \`a des vari\'et\'es projectives et lisses $X/\C$ qui sont ``proches d'\^etre rationnelles'',
en particulier qui sont rationnellement connexes (i.e. telles que $X(\C)/{\rm R}$ soit r\'eduit \`a un point).
C'est le cas  des vari\'et\'es de Fano.

\medskip
 
On a \'etudi\'e :

$\bullet$  les hypersurfaces lisses dans $\P^n_{\C}$ (de degr\'e $d\leq n$)

$\bullet$ les rev\^etements cycliques ramifi\'es de $\P^n_{\C}$ (avec des conditions sur
le degr\'e du rev\^etement et le degr\'e de l'hypersurface de ramification)

$\bullet$ des familles de quadriques de dimension relative $d$ au moins 1
au-dessus de $\P^n_{\C}$ 

$\bullet$ des familles de surfaces de del Pezzo, et plus g\'en\'eralement de vari\'et\'es de Fano,
au-dessus de $\P^n_{\C}$ 

\medskip

On proc\`ede par d\'eg\'en\'erescence de ces vari\'et\'es sur des
vari\'et\'es singuli\`eres $Y/k$,  avec $k$ \'eventuellement de caract\'eristique positive,
pour lesquelles on trouve une
 r\'esolution des singularit\'es $Z \to Y$ qui soit un morphisme $CH_{0}$-trivial,
et l'on montre que $Z$ n'est pas $CH_{0}$-triviale, ou
que $Z$ n'est pas presque ${\rm R}$-triviale 
en utilisant  le groupe de Brauer ou la cohomologie non ramifi\'ee
ou bien,  si le corps r\'esiduel $k$ est de caract\'eristique positive,
l'invariant $H^0(Z,\Omega^{i})$. 

Il y a ici deux points qui demandent beaucoup de travail :

$\bullet$ Montrer que la r\'esolution $Z \to Y$ est  un morphisme $CH_{0}$-trivial
(c'est une propri\'et\'e ind\'ependante de la r\'esolution). En pratique, il faut faire
la r\'esolution explicite, et voir si les fibres sont $CH_{0}$-triviales.

$\bullet$ Montrer qu'un invariant (groupe de Brauer, cohomologie non ramifi\'ee ...)
n'est pas trivial sur $Z$.

\medskip

La premi\`ere m\'ethode, avec $H^{2}_{nr}$, 
alias le groupe de Brauer, est celle qui a \'et\'e utilis\'ee par C. Voisin (doubles solides quartiques)
puis dans \cite{CTP16}
(quartiques lisses dans $\P^4$),
puis par Beauville (doubles solides sextiques),
 et dans de nombreux articles subs\'equents 
de  Hassett, Pirutka, Tschinkel, Kresch,  B\"{o}hning, von Bothmer, Auel.
C'est celle qui a permis le r\'esultat spectaculaire de Hassett, Pirutka, Tschinkel \cite{HPT}
selon lequel la rationalit\'e stable n'est pas forc\'ement constante dans une famille
lisse de dimension relative au moins 4.

 La seconde m\'ethode, avec les diff\'erentielles en caract\'eristique positive,
  a \'et\'e initi\'ee par B. Totaro  \cite{T}.  Elle est inspir\'ee
d'un travail de Koll\'ar de 1995, qui utilisait d\'ej\`a un argument de sp\'ecialisation
sur une vari\'et\'e singuli\`ere en caract\'eristique positive et $H^0(Z,\Omega^{i})$. Totaro en a d\'eduit
des r\'esultats tr\`es g\'en\'eraux sur la non rationalit\'e stable des hypersurfaces tr\`es g\'en\'erales dans $\P^n$,
de degr\'e $d \leq n$ satisfaisant approximativement $d \geq 2n/3$. 
Elle a \'et\'e poursuivie dans  \cite{CTPcycl}.
Des r\'esultats tr\`es g\'en\'eraux  ont \'et\'e ensuite obtenus par cette m\'ethode par T. Okada, 
H.  Ahmadinezhad,  I. Krylov pour d'autres types de vari\'et\'es rationnellement connexes.

La premi\`ere m\'ethode, cette fois-ci avec les invariants cohomologiques
sup\'erieurs $H^{i}_{nr}$, vient d'\^etre utilis\'ee par S. Schreieder \cite{Schr1} pour des
fibrations en quadriques  de grande dimension au-dessus de l'espace projectif.
  \`A cette occasion, il a introduit une variante importante de la m\'ethode de sp\'ecialisation,
qui \'evite dans certains cas de v\'erifier si la r\'esolution de la fibre sp\'eciale
est  $CH_{0}$-triviale (ou presque ${\rm R}$-triviale).

On trouvera ceci discut\'e dans le texte \cite{ctTalca}.

Par des arguments de sp\'ecialisations successives \`a partir du cas des familles de quadriques de Pfister
au-dessus d'un espace projectif, 
Schreieder \cite{Schr3} a fait progresser
de fa\c con spectaculaire le cas des hypersurfaces tr\`es g\'en\'erales de degr\'e $d$
dans $\P^n_{\C}$,
obtenant leur non rationalit\'e stable avec une condition du type $d \geq log(n)$.

\section{Hypersurfaces cubiques non stablement rationnelles sur un corps non alg\'ebriquement clos}

Soit $k$ un corps et $X \subset \P^n_{k}$, $n \geq 3$, une hypersurface cubique lisse.
On s'int\'eresse ici au cas o\`u $k$ n'est pas alg\'ebriquement clos.

Le d\'efi ici est, pour un corps $k$ de complexit\'e arithm\'etique donn\'e
(corps fini, corps local, corps de nombres, corps de fonctions de $d$ variables 
sur un de ces corps ou sur les complexes, corps de s\'eries formelles it\'er\'ees
sur l'un de ces corps) de trouver des 
 hypersurfaces cubiques lisses non r\'etractilement rationnelles $X \subset \P^n_{k}$ avec $X(k) \neq \emptyset$ 
 et $n$ aussi grand que possible.

\subsection{Hypersurfaces cubiques r\'eelles} 

\begin{prop}
Pour tout entier $n \geq  2$, il existe une hypersurface cubique lisse
$X \subset \P^n_{\R}$ telle que le lieu des points r\'eels $X(\R)$ ait deux
composantes connexes.  En particulier,
une telle hypersurface n'est pas r\'etractilement rationnelle.
\end{prop}

\begin{proof}
Soit $n\geq 2$, soit  $x_{0}, \dots,x_{n-2}, u,v$ des variables,
soit 
$$ \Phi(x_{0}, \dots,x_{n-2},u,v) = (\sum_{i} x_{i}^2) v - u(u-v)(u+v).$$
Soit 
 $Y \subset \P^n_{\R}$ 
 l'hypersurface cubique d\'efinie par
l'\'equation
$$ \Phi(x_{0}, \dots,x_{n-2},u,v) =0.$$
Son lieu singulier est donn\'e par $u=v=\sum_{i} x_{i}^2=0$, il n'a pas de point r\'eel. 
On a donc $Y_{lisse}(\R)=Y(\R)$. Les coordonn\'ees $(u,v)$ d\'efinissent une 
application continue $Y_{lisse}(\R) \to \P^1(\R)$, dont l'image est la r\'eunion des
deux invervalles d\'efinis par $u(u-v)(u+v)\geq  0$.
On v\'erifie ainsi que $Y(\R)$ est une vari\'et\'e $C^{\infty}$ avec deux composantes
connexes.
Soit $\Psi(x_{0}, \dots,x_{n-2},u,v) = \sum_{i} x_{i}^3 +u^3+v^3$.
Pour $\epsilon \in \R$ petit, l'hypersurface cubique d\'efinie par $\Phi+ \epsilon \Psi=0$
est lisse pour $\epsilon \neq 0$, pout tout $\epsilon \in \R$ petit, son lieu r\'eel est 
une vari\'et\'e $C^{\infty}$ lisse, et par le th\'eor\`eme d'Ehresmann, ce lieu est
diff\'eomorphe \`a $Y(\R)=Y_{lisse}(\R)$.
\end{proof}

Exercice:  pour  une hypersurface cubique $X \subset \P^n_{\R}$, l'espace $X(\R)$
a au plus deux composantes connexes.

 \subsection{Sp\'ecialisations \`a fibres r\'eductibles}
 Dans le contexte de la sp\'ecialisation du groupe de Chow, Totaro \cite{T}
 a utilis\'e des sp\'ecialisations \`a fibre r\'eductible.
 On peut le faire aussi dans le cadre de la ${\rm R}$-\'equivalence.
 L'\'enonc\'e suivant est inspir\'e par \cite{T} et \cite{ChL}, mais est plus simple.

\begin{prop}\label{fibrereductible}
Soit $R$ un anneau de valuation discr\`ete, $K$ son corps des fractions, $k$
son corps r\'esiduel.  Soit $\sX$ un $R$-sch\'ema propre et plat. Supposons la fibre g\'en\'erique $X/K$ 
lisse et g\'eom\'etriquement int\`egre.
Soit $Y$ la fibre sp\'eciale.  Supposons  $Y$ union de deux ferm\'es
 $Y= V \cup W$, $T=V \cap W$, 
$T(k)=\emptyset$, $V_{lisse}(k) \neq \emptyset$ et $W_{lisse}(k) \neq \emptyset$.
Alors la $K$-vari\'et\'e $X$ n'est pas   ${\rm R}$-triviale et n'est donc pas
r\'etractilement rationnelle.
\end{prop}

\begin{proof}
On peut supposer  que $R$ est hens\'elien. Par le lemme de Hensel, on trouve
des $R$-points $A$ et $B$ de $\sX(R)=X(K)$ qui se sp\'ecialisent l'un dans $V(k)$,
l'autre dans $W(k)$. Par le th\'eor\`eme \ref{specialisationR},  l'application de sp\'ecialisation
$X(K)=\sX(R) \to Y(k)$ passe au quotient par la ${\rm R}$-\'equivalence.
Il existe donc
 un $k$-morphisme $f : \P^1_{k} \to Y$ tel que $f(0) \in V(k)$ et
$f(\infty) \in W(k)$. La courbe $\P^1_{k}$ est alors couverte par les deux ferm\'es
non vides $v=f^{1}(V)$ et $w=f^{-1}(W)$, qui contiennent chacun un $k$-point,
et dont l'intersection n'a pas de $k$-point.
L'un des deux ferm\'es, soit $v$ est \'egal \`a $\P^1_{k}$. Mais alors
$w \subset v$, et tout $k$-point de $w$ est dans $v$. Contradiction.
\end{proof}

  \medskip
  
\begin{ex}  Soit $n \geq 2$ et $f_{0}(x_{1}, \dots, x_{n}) \in k[x_{1}, \dots, x_{n}]$ 
une forme homog\`ene   de degr\'e $d\geq 2$ sans z\'ero sur le corps $k$ d\'efinissant une hypersurface lisse sur $k$.
 Soit $\alpha \in k^*$ une valeur (non nulle) de $f$ sur $k^n$.
  Soit $f(x_{0}, \dots, x_{n}): = \alpha x_{0}^d - f_{0}(x_{1}, \dots, x_{n}) \in k[x_{0},x_{1}, \dots, x_{n}]$,
  puis  $g(x_{0}, \dots, x_{n})= x_{0} . f(x_{0} , \dots, x_{n})$.
  Soit $Y \subset \P^n_{k}$ l'hypersurface de degr\'e $d+1$ d\'efinie par $g=0$. C'est l'union de $V$
  d\'efini par $x_{0}=0$ et $W$ d\'efini par $f_{0}(x_{1}, \dots, x_{n})=0$. L'intersection $T=V \cap W$
  satisfait $T(k)=\emptyset$. On a $V(k) \neq \emptyset$ et $W(k) \neq \emptyset$.
  
  Soit $g(x_{0}, \dots, x_{n}) \in k[x_{0},x_{1}, \dots, x_{n}]$ une forme homog\`ene de degr\'e $d+1$
  d\'efinissant une hypersurface lisse dans $\P^n_{k}$.
  On consid\`ere alors $R=k[[t]]$, $K=k((t))$. L'hypersurface $X \subset \P^n_{K}$
  d\'efinie par $tg(x_{0}, \dots, x_{n})+f(x_{0}, \dots, x_{n})=0$ n'est pas ${\rm R}$-triviale,
  et n'est pas pas r\'etractilement rationnelle. 
  
  On peut aussi donner des exemples similaires avec $R$ un anneau de valuation discr\`ete complet
  d'in\'egale caract\'eristique.
  
  En utilisant cette m\'ethode dans le cas $d=2$, on obtient des hypersurfaces cubiques lisses, avec un $K$-point,  non r\'etractilement rationnelles
 sur $\P^{N}_{K}$ pour tout $N \leq   2^{r-1}$  sur    $K=\C((u_{1}))\dots ((u_{r}))$ et dans  $\Q_{p}((u_{1})) \dots ((u_{r-2}))$.
 
 Sur   $K=\C((u_{1}))((u_{2}))((u_{3}))$ on trouve donc des hypersurfaces cubiques dans $\P^4_{K}$.
Ces bornes sont les m\^emes que celles obtenues dans 
  \cite{ChL} et dans \cite{ctmumbai}, qui \'etablissent le r\'esultat plus fort que
  les hypersurfaces cubiques concern\'ees ne sont pas $CH_{0}$-triviales.
 La d\'emonstration de ce dernier r\'esultat   utilise une variation due \`a Totaro de la technique de 
  sp\'ecialisation de Voisin et CT-Pirutka pour les groupes de Chow de z\'ero-cycles.

\end{ex}

 \bigskip
 
\subsection{Hypersurfaces cubiques diagonales et cohomologie non ramifi\'ee}\label{diagonal}

Ce paragraphe est extrait directement de l'article \cite{ctmumbai}.  On utilise encore ici
une technique de sp\'ecialisation, mais elle est diff\'erente de celles employ\'ees ci-dessus.

\begin{theo}\label{theodiagonal}
Soit $k$ un corps de caract\'eristique diff\'erente de 3, poss\'edant un \'el\'ement $a$
qui n'est pas un cube. Soient $0 \leq n \leq m$ des entiers.
Soit $F$ un corps   avec
$$k(\lambda_{1}, \dots, \lambda_{m}) \subset F \subset F_{m}:=k((\lambda_{1})) \dots ((\lambda_{m})).$$
L'hypersurface cubique      $X:= X_{n,F}$ de $\P^{n+3}_{F}$ d\'efinie par l'\'equation
$$ x^3+y^3+z^3+aw^3+ \sum_{i=1}^n \lambda_{i} t_{i}^3=0$$
 poss\`ede un point rationnel et 
n'est pas
 universellement $CH_{0}$-triviale,  en particulier elle  n'est pas r\'etractilement rationnelle.
\end{theo}

\begin{proof} 
Pour \'etablir le r\'esultat, on peut supposer que $k$ contient une racine cubique primitive de l'unit\'e, soit $j$,
et que $F=F_{m}$.
Le lemme  \ref{k((t))}  ci-dessous permet de supposer $n=m$.
On fixe un isomorphisme $\Z/3 = \mu_{3}$ et on consid\`ere la cohomologie \'etale 
\`a coefficients $\Z/3$.  On ignore les torsions \`a la Tate dans les notations.
Etant donn\'es un corps $L$ contenant $k$ et des \'el\'ements $b_{i}, i=1, \dots, s,$ de $L^*$,
on note 
$(b_{1}, \dots, b_{s}) \in H^s(L,\Z/3)$ le cup-produit, en cohomologie galoisienne,
des classes $(b_{i}) \in L^*/L^{*3}=H^1(L,\Z/3)$.

 On va d\'emontrer par
r\'ecurrence sur $n\neq 0$  l'assertion suivante, qui implique la proposition.

($A_{n}$) 
 Soient $k$, $a$, $F_{n}$ et $X_{n}/F_{n}$ comme ci-dessus.   
Le cup-produit
$$\alpha_{n}: = ((x+jy)/(x+y), a, \lambda_{1},\dots,\lambda_{n}) \in H^{n+2}(F_{n}(X_{n}),\Z/3)$$
d\'efinit une classe de cohomologie non ramifi\'ee  (par rapport au corps de base $F_{n}$) qui ne provient pas
d'une classe dans $H^{n+2}(F_{n},\Z/3)$.

Le cas $n=0$ est connu (\cite[Chap. VI, \S 5]{manin} \cite[\S 2.5.1]{CTSa87a}).
Supposons l'assertion d\'emontr\'ee pour $n$.

La classe $\alpha_{n+1}$ sur la $F_{n+1}$-hypersurface $X_{n+1} \subset \P^{n+4}_{F_{n+1}}$ a ses r\'esidus 
triviaux en dehors des diviseurs d\'efinis par $x+y=0$ et $x+jy=0$.
Soit $\Delta \subset X_{n+1}$ le diviseur $x+y=0$. Ce diviseur est d\'efini par les
\'equations 
$$x+y=0, z^3+aw^3+ \sum_{i=1}^{n+1} \lambda_{i} t_{i}^3=0.$$
Le r\'esidu  de $\alpha_{n+1}$ au point g\'en\'erique de $\Delta$ est
 $$\partial_{\Delta}(\alpha_{n+1}) = 
  \pm (a, \lambda_{1},\dots,\lambda_{n+1}) \in H^{n+2}(F_{n+1}(\Delta),\Z/3).$$
 Mais dans le corps des fonctions de $\Delta$, on a 
$$1+a(w/z)^3+ \sum_{i=1}^{n+1} \lambda_{i} (t_{i}/z)^3=0$$
et 
cette \'egalit\'e implique (cf. \cite[Lemma 1.3]{milnor})  :
$$(a,\lambda_{1}, \dots, \lambda_{n+1}) = 0 \in  H^{n+2}(F_{n+1}(\Delta),\Z/3).$$
Le m\^{e}me argument s'applique pour le diviseur
d\'efini par  $x+jy=0$. Ainsi $\alpha_{n+1}$ est une classe de cohomologie non
ramifi\'ee sur la $F_{n+1}$-hypersurface $X_{n+1}$.

Soit   $\mathcal{X}_{n+1}$ le $F_{n}[[\lambda_{n+1}]]$-sch\'ema d\'efini par
$$ x^3+y^3+z^3+aw^3+ \sum_{i=1}^{n+1} \lambda_{i} t_{i}^3=0.$$
Le diviseur $Z$ d\'efini par $\lambda_{n+1}=0$ sur $\mathcal{X}$
est le c\^{o}ne d'\'equation
$$ x^3+y^3+z^3+aw^3+ \sum_{i=1}^n \lambda_{i} t_{i}^3=0$$
dans $\P^{n+4}_{F_{n}}$, c\^one qui est birationnel au produit de $\P^1_{F_{n}}$
et de l'hypersurface cubique lisse  $X_{n} \subset \P_{F_{n}}^{n+3}$ d\'efinie par
$$ x^3+y^3+z^3+aw^3+ \sum_{i=1}^n \lambda_{i} t_{i}^3=0.$$
 Le corps des fonctions rationnelles de $\mathcal{X}_{n+1}$ est
$F_{n+1}(X_{n+1})$.

On a  $$\partial_{Z}(\alpha_{n+1} ) =   \pm  ((x+jy)/(x+y), a, \lambda_{1},\dots,\lambda_{n})  \in H^{n+2}(F_{n}(Z), \Z/3).$$

Par l'hypoth\`ese de r\'ecurrence 
$$ ((x+jy)/(x+y), a, \lambda_{1},\dots,\lambda_{n}) \in H^{n+2}(F_{n}(X_{n}), \Z/3)$$
n'est pas dans l'image de $ H^{n+2}(F_{n},\Z/3)$. Ceci implique que
$$ ((x+jy)/(x+y), a, \lambda_{1},\dots,\lambda_{n}) \in H^{n+2}(F_{n}(Z)), \Z/3)$$
n'est pas dans l'image de $H^{n+2}(F_{n},\Z/3)$.
Du  diagramme commutatif 
$$\begin{array}{ccccccccc}
 \partial_{Z} : & H^{n+3}(F_{n+1}(X),\Z/3)  & \to & H^{n+2}(F_{n}(Z), \Z/3)\\
 & \uparrow & &  \uparrow \\
 \partial_{\lambda_{n+1}=0} : & H^{n+3}(F_{n+1},\Z/3) & \to  &H^{n+2}(F_{n},\Z/3)
 \end{array}$$
on conclut que 
$$\alpha_{n+1}: = ((x+jy)/(x+y), a, \lambda_{1},\dots,\lambda_{n+1}) \in H^{n+3}(F_{n+1}(X),\Z/3)$$
n'est pas dans l'image de $H^{n+3}(F_{n+1},\Z/3)$.

Ceci \'etablit $(A_{n})$ pour tout entier $n$ et
 implique (cf. \cite{merk}) que la  $F_{n}$-vari\'et\'e $X_{n}$ n'est pas universellement $CH_{0}$-triviale  et   n'est pas r\'etractilement rationnelle.
\end{proof}

\begin{lem}\label{k((t))}
Soit $F$ un corps. Si  une $F$-vari\'et\'e $X$
projective, lisse, g\'eom\'etriquement connexe n'est pas universellement $CH_{0}$-triviale, alors la $F((t))$-vari\'et\'e 
$X\times_{F}F((t))$ n'est pas universellement $CH_{0}$-triviale, et donc n'est pas r\'etractilement rationnelle.
\end{lem}
\begin{proof}
Sur tout corps $L$ contenant $F$, on dispose de l'application de sp\'ecialisation
$CH_{0}(X_{L((t))}) \to CH_{0}(X_{L})$, et cette application est surjective et respecte
le degr\'e.
\end{proof}

\begin{rema}
Il serait int\'eressant de comprendre la g\'en\'eralit\'e de la construction faite dans le
th\'eor\`eme \ref{theodiagonal}. On utilise une classe de cohomologie non ramifi\'ee non constante sur un mod\`ele
birationnel de la fibre sp\'eciale d'une $k[[t]]$-sch\'ema propre \`a fibres int\`egres, et on en tire une classe de cohomologie non ramifi\'ee non constante
de degr\'e un de plus sur la fibre g\'en\'erique sur $k((t))$, essentiellement par cup-produit avec la classe d'une uniformisante
de l'anneau $k[[t]]$.
\end{rema}

On laisse au lecteur le soin d'\'etablir l'analogue suivant du th\'eor\`eme \ref{theodiagonal}.

\begin{theo}\label{theodiagonalsurpadique}
Soient $p \neq 3$ un nombre premier et $k$ un corps $p$-adique dont le corps r\'esiduel
contient les racines cubiques primitives de 1. Soit $a \in k^*$ une unit\'e qui n'est pas un cube.
Soit $\pi$ une uniformisante de $k$. Soient $0 \leq n \leq m$ des entiers.
Soit $F$ un corps   avec
$$\Q(a) (\lambda_{1}, \dots, \lambda_{m}) \subset F \subset k((\lambda_{1})) \dots ((\lambda_{m})).$$
L'hypersurface cubique      $X_{n}$ de $\P^{n+4}_{F}$ d\'efinie par l'\'equation
$$ x^3+y^3+z^3+aw^3+ \pi t^3+  \sum_{i=1}^n \lambda_{i} t_{i}^3=0,$$
qui poss\`ede un point rationnel,
n'est pas
 universellement $CH_{0}$-triviale et  donc n'est pas r\'etractilement rationnelle.
\end{theo}

\bigskip

{\it Exemples}

 En appliquant le th\'eor\`eme \ref{theodiagonal},
on trouve  $X_{n} \subset \P^{n+3}_{F}$ non r\'etractilement rationnelle
avec 
$$k(\lambda_{1}, \dots, \lambda_{n}) \subset F \subset k((\lambda_{1})) \dots ((\lambda_{n}))$$
dans les situations suivantes.

(i) Le corps $k=\F$ 
est un corps fini 
de caract\'eristique diff\'erente de 3
contenant les racines cubiques de 1.

(ii) Le  corps $k$, de caract\'eristique diff\'erente de 3, poss\`ede une valuation discr\`ete,
par exemple $k$ est le  corps des fonctions d'une vari\'et\'e complexe de dimension au moins 1, ou est un corps $p$-adique, ou est un corps de nombres.  

  On trouve ainsi des hypersurfaces cubiques lisses non r\'etractilement rationnelles
dans $\P_{\C(x_{1},\dots,x_{m})}^{n}$, avec un point rationnel, pour tout  entier $n$ avec $3 \leq n \leq m+2$.

En appliquant le th\'eor\`eme \ref{theodiagonalsurpadique}, 
sur un corps $k$ $p$-adique ($p\neq 3$) contenant une racine cubique de $1$,    on trouve
des hypersurfaces cubiques lisses non r\'etractilement rationnelles
dans 
$\P_{k(x_{1},\dots,x_{m})  }^{n}$, avec un point rationnel, pour tout  entier $n$  avec
$4 \leq n \leq m+4$.

\end{document}